\documentclass[a4paper,11pt,pdf]{amsart}

\usepackage{enumerate, amsmath, amsfonts, amssymb, amsthm, thmtools, wasysym, graphics, graphicx, xcolor, frcursive,comment,bbm}

\usepackage{etex}

\usepackage{vmargin}            
\setmarginsrb{3.3cm}{3.2cm}{3.3cm}{2.8cm}{0cm}{0.6cm}{0cm}{0cm}

\usepackage[all]{xy}
\usepackage[T1]{fontenc}

\usepackage{array}

\makeatletter
\newcommand{\thickhline}{%
    \noalign {\ifnum 0=`}\fi \hrule height 1pt
    \futurelet \reserved@a \@xhline
}
\newcolumntype{"}{@{\hskip\tabcolsep\vrule width 1pt\hskip\tabcolsep}}
\makeatother

\def\NC{{\sf{NC}}}

\def\Ab{\widetilde{A}_{B_n}}

\newcommand\bt{{\mathbf{t}}}

\newcommand\bS{{\mathbf{S}}}

\def\lt{{{\ell_T}}}

\usepackage{etex}
\newcommand\bs{{\mathbf{s}}}

\definecolor{darkblue}{rgb}{0.0,0,0.7} 
\newcommand{\darkblue}{\color{darkblue}} 
\definecolor{darkred}{rgb}{0.7,0,0} 
\newcommand{\defn}[1]{\emph{\darkblue #1}} 

\usepackage{hyperref}
\usepackage[all]{xy}
\usepackage[T1]{fontenc}

\usepackage[nameinlink]{cleveref}
\usepackage[all]{xy}
\usepackage[T1]{fontenc}

\usepackage[colorinlistoftodos]{todonotes}

\def\NC{{\sf{NC}}}

\def\supp{{\sf{supp}}}

\newtheorem{theorem}{Theorem}[section]
\newtheorem{proposition}[theorem]{Proposition}
\newtheorem{lemma}[theorem]{Lemma}
\newtheorem{corollary}[theorem]{Corollary}

\theoremstyle{definition}
\newtheorem{definition}[theorem]{Definition}
\newtheorem{rmq}[theorem]{Remark}
\newtheorem{exple}[theorem]{Example}

\theoremstyle{remark}

\usepackage{graphicx}                  
\usepackage{pstricks,pst-plot,pst-text,pst-tree,pst-eps,pst-fill,pst-node,pst-math}
\usepackage{setspace}

\title{Simple dual braids, noncrossing partitions and Mikado braids of type $D_n$}

\author{Barbara Baumeister}
\address{Barbara Baumeister, Fakult\"{a}t f\"{u}r Mathematik,
Universit\"{a}t Bielefeld,
Postfach 10 01 31,
33501 Bielefeld, Germany.}
\email{b.baumeister@math.uni-bielefeld.de}
\author{Thomas Gobet}\thanks{The second author was partially funded by the ANR Geolie ANR-15-CE40-0012}
\address{Thomas Gobet, Institut \'Elie Cartan de Lorraine, Universit\'e  de Lorraine, site de Nancy,
B.P. 70239,
54506 Vandoeuvre-l\`{e}s-Nancy Cedex, France}
\email{thomas.gobet@univ-lorraine.fr}

\begin{document}
\maketitle
\begin{abstract}
We show that the simple elements of the dual Garside structure of an Artin group of type $D_n$ are Mikado braids, giving a positive answer to a conjecture of Digne and the second author. To this end, we use an embedding of the Artin group of type $D_n$ in a suitable quotient of an Artin group of type $B_n$ noticed by Allcock, of which we give a simple algebraic proof here. This allows one to give a characterization of the Mikado braids of type $D_n$ in terms of those of type $B_n$ and also to describe them topologically. Using this topological representation and Athanasiadis and Reiner's model for noncrossing partitions of type $D_n$ which can be used to represent the simple elements, we deduce the above mentioned conjecture. 
\end{abstract}
  
\medskip
~\\
\noindent\textbf{AMS 2010 Mathematics Classification}: : ~20F36, ~20F55.\\


\tableofcontents
\thispagestyle{empty}

\section{Introduction}

 The dual braid monoid $B_c^*$ of a Coxeter system $(W,S)$ of spherical type was introduced by Bessis \cite{Dual} and depends on the choice of a standard Coxeter element $c\in W$ (a product of all the elements of $S$ in some order). It is generated by a copy $T_c$ of the set $T$ of reflections of $W$, that is, elements which are conjugates to elements of $S$. As a Garside monoid, it embeds into its group of fractions, which was shown by Bessis to be isomorphic to the Artin group $A_W$ 
 corresponding to $W$. Unfortunately, this isomorphism is poorly understood, and the proof of its existence requires a case-by-case argument \cite[Fact 2.2.4]{Dual}. 

The aim of this note is to study properties of the simple elements $\mathrm{Div}(c)$ in $B_c^*$ viewed inside $A_W$ in case $W$ is of type $D_n$ and to show that they are \textit{Mikado braids}, that is, that they can be represented as a quotient of two positive canonical lifts of elements of $W$. These braids appeared in work of Dehornoy \cite{D} in type $A_n$ and in work of Dyer \cite{Dyernil} for arbitrary Coxeter systems and have many interesting properties. For example, they satisfy an analogue of Matsumoto's Lemma in Coxeter groups \cite[Section 9]{Dyernil}. We refer the reader to \cite[Section 9]{Dyernil}, \cite[Section 4]{DG} (there the Mikado braids are called \textit{rational permutation braids}, while the terminology Mikado braids rather refers to braids viewed topologically; it is shown however in \cite{DG} that both are equivalent) or \cite[Section 3.2]{Twisted} for more on the topic. Another important property is that their images in the Iwahori-Hecke algebra $H(W)$ of the Coxeter system $(W,S)$ have positivity properties; let us be more precise. There is a natural group homomorphism $a: A_W\longrightarrow H(W)^\times$. If $\beta\in A_W$ is a Mikado braid and if we express its image $a(\beta)$ in the canonical basis $\{C_w~|~w \in W\}$ of the Hecke algebra, then the coefficients are Laurent polynomials with positive coefficients (see \cite[Section 8]{DG}). This is one of the main motivations for studying Mikado braids, and showing that simple dual braids are Mikado braids. This last property was conjectured for an arbitrary Coxeter system $(W,S)$ of spherical type in \cite{DG}, and shown to hold in all the irreducible types different from $D_n$ \cite[Theorems 5.12, 6.6, 7.1]{DG}. 

In the classical types $A_n$ and $B_n$, the conjecture is proven using a topological characterization of Mikado braids: it can be seen on any reduced braid diagram (resp. symmetric braid diagram in type $B_n$) whether a braid is a Mikado braid or not. The present paper gives topological models for Mikado braids of type $D_n$, similar to those given in types $A_n$ and $B_n$ in \cite{DG}, and solves the above conjecture in the remaining type $D_n$:

\begin{theorem}\label{theorem:main}
Let $c$ be a standard Coxeter element in a Coxeter group $(W, S)$ of type $D_n$. Then every element of $\mathrm{Div}(c)$ is a Mikado braid. 
\end{theorem}

 As a consequence, every simple dual braid in every spherical type Artin group is a Mikado braid, the reduction to the irreducible case being immediate. Licata and Queffelec recently informed us that they also have a proof of the conjecture in types $A, D, E$ with a different approach using categorification \cite{QL}.

To prove the conjecture, we proceed as follows. Firstly, we explicitly realize the Artin group $A_{D_n}$ of type $D_n$ as an index two subgroup of a quotient of the Artin group $A_{B_n}$ of type $B_n$. The existence of such a realization, which is of independent interest, is not new: it was noticed by Allcock \cite[Section 4]{Allcock}.
We give a simple proof of it here (Proposition~\ref{prop:quotient}). This allows to realize elements of $A_{D_n}$ topologically by Artin braids. We then characterize Mikado braids of type $D_n$ as the images of those Mikado braids of type $B_n$ which surject onto elements of $W_{D_n}\subseteq W_{B_n}$ under the canonical map from $A_{B_n}$ onto $W_{B_n}$ (Theorem~\ref{thm:mikado_bd}). This implies that Mikado braids of type $D_n$ satisfy a nice topological condition, and gives a model for their study in terms of symmetric Artin braids, because elements of $A_{B_n}$ can be realized as symmetric Artin braids on $2n$ strands (see Section~\ref{typeB}). Using Athanasiadis and Reiner's graphical model \cite{AR} for $c$-noncrossing partitions of type $D_n$ (which are in canonical bijections with the simple elements $\mathrm{Div}(c)$ of $B_c^*$; we denote this bijection by $x\mapsto x_c$, where $x$ is a $c$-noncrossing partition), we attach to every such noncrossing partition $x$ an Artin braid $\beta_x$ of type $B_n$, whose image in the above mentioned quotient is precisely the element $x_c\in A_{D_n}$ (Section~\ref{main}). Using the topological characterization of Mikado braids of type $B_n$ from~\cite{DG}, we then prove that $\beta_x$ is a Mikado braid of type $B_n$ (Proposition~\ref{beta_mik}), which concludes by the above mentioned characterization of Mikado braids of type $D_n$ (Theorem~\ref{thm:main}).\\    
~\\
\textbf{Acknowledgments.} We thank Luis Paris for useful discussions with the second author and Jon McCammond for pointing out the reference~\cite{MS}.

\section{Artin groups of type $D_n$ inside quotients of Artin groups of type $B_n$}

\subsection{Coxeter groups and Artin groups}

This section is devoted to recalling basic facts on Coxeter groups and their Artin groups. We refer the reader to  \cite{Bou, Hum} or \cite{BjBr} for more on the topic. A \defn{Coxeter system} $(W,S)$ is a group $W$ generated by a set $S$ of involutions subject to additional \defn{braid relations}, that is, relations of the form $st\cdots = ts\cdots$ for $s,t\in S$, $s\neq t$. Here $st\cdots$ denotes a strictly alternating product of $s$ and $t$, and the number $m_{st}$ of factors in the left hand side equals the number $m_{ts}$ of factors in the right hand side. We have $m_{st}\in\{2, 3,\dots\}\cup\{\infty\}$, the case $m_{st}=\infty$ meaning that there is no relation between $s$ and $t$. Let $\ell:W\rightarrow \mathbb{Z}_{\geq 0}$ be the length function with respect to the set of generators $S$. 

Finite irreducible Coxeter groups are classified in four infinite families of types $A_n$, $B_n$, $D_n$, $I_2(m)$ and six exceptional groups of types $E_6, E_7, E_8, F_4, H_3, H_4$. If $X$ is a given type, we denote by $(W_X, S_X)$ a Coxeter system of this type. 

The \defn{Artin group} $A_W$ attached to the Coxeter system $(W,S)$ is generated by a copy $\bS$ of the elements of $S$, subject only to the braid relations. This gives rise to a canonical surjection $\pi: A_W\twoheadrightarrow W$ induced by $\bs\mapsto s$. If $W$ has type $X$, we simply denote $A_W$ by $A_X$. 

The canonical map $\pi$ has a set-theoretic section $W\hookrightarrow A_W$ built as follows: let $w=s_1 s_2\cdots s_k$ be a reduced expression for $w$, that is, we have $s_i\in S$ for all $i=1,\dots, k$ and $k=\ell(w)$. Then the lift $\bs_1 \bs_2\cdots \bs_k$ in $A_W$ is independent of the chosen reduced expression, and we therefore denote it by $\mathbf{w}$. This is a consequence of the fact that in every Coxeter group, one can pass from any reduced expression of a fixed element $w$ to any other just by applying a sequence of braid relations. The element $\mathbf{w}$ is the \defn{canonical positive lift} of $w$.     

\subsection{Embeddings of Coxeter groups}\label{emb:cox}
Let $(W_{B_n}, S_{B_n})$ be a Coxeter system of type $B_n$. We will identify it with the signed permutations group as follows: let $S_{-n,n}$ be the group of permutations of $[-n,n]=\{-n,-n+1,\dots, -1,1,\dots, n\}$ and define
$$W_{B_n}:=\{w\in S_{-n,n}~|~w(-i)=-w(i),~\mbox{for all}~ i\in[-n,n]\}.$$
Then setting $s_0:=(-1, 1)$ and $s_i=(i, i+1)(-i, -i-1)$ for all $i=1, \dots, n-1$ we get that $S_{B_n}=\{s_0, s_1,\dots, s_{n-1}\}$ is a simple system for $W_{B_n}$ (see \cite[Section 8.1]{BjBr}).

Let $(W_{D_n}, S_{D_n})$ be a Coxeter group of type $D_n$. Recall that $W_{D_n}$ can be realized as an index two subgroup of $W_{B_n}$ as follows: setting $t_0=s_0 s_1 s_0$, $t_i=s_i$ for all $i=1,\dots, n-1$ we have that $S_{D_n}:=\{t_0, t_1,\dots, t_{n-1}\}$ is a simple system for the Coxeter group $W_{D_n}=\left\langle t_0, t_1,\dots, t_{n-1}\right\rangle$ of type $D_n$ (see \cite[Section 8.2]{BjBr}). In the following, a Coxeter group of type $D_n$ will always be viewed inside $W_{B_n}$, with the above identifications. 

\subsection{Embeddings of Artin groups}

We assume the reader to be familiar with Artin groups attached to Coxeter groups and refer to \cite[Chapter IX]{Garside} for basic results. Notice that there are two surjective maps $q_B: A_{B_n}\longrightarrow A_{A_{n-1}}$, $q_D: A_{D_n}\longrightarrow A_{A_{n-1}}$ defined as follows: if we denote by $\{\sigma_1, \dots, \sigma_{n-1}\}$ the set of standard Artin generators of the $n$-strand Artin braid group $A_{A_{n-1}}$, then $q_B(\mathbf{s}_0)=1$, $q_B(\mathbf{s}_i)=\sigma_i$ for $i\neq 0$, while $q_D(\mathbf{t}_0)=\sigma_1$, $q_D(\mathbf{t}_i)=\sigma_i$ for all $i\neq 0$ (see~\cite[Section 2.1]{CP}). Both maps $q_B$ and $q_D$ are split and one can write $A_{X_n}\cong \mathrm{ker}(q_X)\rtimes A_{A_{n-1}}$ for $X\in\{B,D\}$.  

Crisp and Paris showed that the embedding of $W_{D_n}$ in $W_{B_n}$ which we recalled in Subsection~\ref{emb:cox} does not come from an embedding $\varphi: A_{D_n}\longrightarrow A_{B_n}$ such that $q_D=q_B\circ\varphi$ \cite[Proposition 2.6]{CP}. In this section we show that there is an embedding of $A_{D_n}$ inside a quotient $\widetilde{A}_{B_n}$ of $A_{B_n}$; this embedding can be seen as a natural lift of the embedding of Coxeter groups and has the expected properties (see Lemma~\ref{lem:comp}). This is mostly a reformulation of results of Allcock \cite[Sections 2 and 4]{Allcock}, but we will give a simple algebraic proof of this fact here.  

\begin{definition}
Define $\Ab$ to be the quotient of $A_{B_n}$ by the smallest normal subgroup containing $\bs_0^2$. 
\end{definition}

It follows immediately from this definition that the canonical map $\pi_n: A_{B_n}\twoheadrightarrow W_{B_n}$ factors through $\Ab$ via two surjective maps $\pi_{n,1}: A_{B_n}\twoheadrightarrow \Ab$ and $\pi_{n,2}: \Ab\twoheadrightarrow W_{B_n}$. 

\begin{rmq}
In \cite[Definition 3.3]{MS}, a similar group, called the \textit{middle group}, is considered. It is defined as the quotient of $A_{B_n}$ by the smallest normal subgroup containing $\bs_1^2$ (as a consequence, every $\bs_i^2$ for $i\geq 1$ is equal to $1$ in the quotient since $\bs_i$ lies in the same conjugacy class as $\bs_1$).  
\end{rmq}

Denote by $s_i'$, $i=0,\dots, n-1$ the image of $\mathbf{s}_i\in A_{B_n}$ in $\Ab$, for all $s_i\in S_{B_n}$. Set $t_0'= s_0' s_1' s_0'$ and $t_i'= s_i'$ for $i=1, \dots, n-1$. 

\begin{lemma}\label{RelationsDn}
The elements $t_0', t_1',\dots t_{n-1}'$ satisfy the braid relations of type $D_n$, that is, we have $$t_0' t_1'=t_1' t_0', ~t_0' t_2' t_0'= t_2' t_0' t_2', ~t_i ' t_{i+1}' t_i'= t_{i+1}' t_i' t_{i+1}'~\mbox{for all}~ i=1, \dots, n-2,$$  $$t_i't_j'=t_j' t_i'\text{ if }|i-j|>1\text{ and }\{i,j\}\neq\{0, 2\}.$$
\end{lemma}
\begin{proof}
All the relations except the second one are immediate consequences of the type $B_n$ braid relations satisfied by the $s_0', s_1', \dots, s_{n-1}'$. For the second relation we have 
\begin{eqnarray*}
t_0' t_2' t_0' &=& s_0' s_1' s_0' s_2' s_0' s_1' s_0'=s_0' s_1' s_0'^2 s_2' s_1' s_0'= s_0' s_1' s_2' s_1' s_0'=s_0' s_2' s_1' s_2' s_0'\\
&=& s_2' s_0' s_1' s_0' s_2'= t_2' t_0' t_2'.
\end{eqnarray*} 

\end{proof}
An immediate corollary is
\begin{corollary}
There is a group homomorphism $\iota_n: A_{D_n}\longrightarrow \Ab$ defined by $\iota_n(\mathbf{t}_i)=t_i'$ for all $i=0, \dots, n-1$. 
\end{corollary}

We have the following situation

\begin{lemma}\label{lem:diag}

There is a commutative diagram

$$\xymatrix{
A_{B_n} \ar@{->>}[rd]_{\pi_{n}^B} \ar@{->>}[r]^{\pi_{n,1}} & \relax \Ab  \ar@{->>}[d]^{\pi_{n,2}} & \langle t_0', \dots, t_{n-1}'\rangle \ar@{^{(}->}[l] \ar@{->>}[d]_{\pi_{n}^D} & A_{D_n} \ar@{->>}[l] \\  
	~ & W_{B_n} & W_{D_n} \ar@{^{(}->}[l] & ~
  }$$

\noindent where $\pi_n^D: \langle t_0', \dots, t_{n-1}'\rangle \longrightarrow W_{D_n}$ is defined by $\pi_n^D(t_i')=t_i$ for all $i=0,\dots, n-1$.  

\end{lemma}

\begin{rmq}
In Proposition~\ref{prop:quotient} below we will show that the map $\iota_n$ is injective; hence $\pi_n^D$ is in fact simply the canonical surjection $A_{D_n}\twoheadrightarrow W_{D_n}$. 
\end{rmq}

\begin{proof}
We have to show that the composition of $\pi_{n,2}$ and $\langle t_0',\dots, t_{n-1}'\rangle\hookrightarrow \Ab$ factors through $W_{D_n}$. It suffices to show that the image of $t_i'$ under this composition is precisely $t_i$ (viewed inside $W_{B_n}$ via the embedding $W_{D_n}\hookrightarrow W_{B_n}$) for all $i=0,\dots, n-1$, which is immediate. 
\end{proof}

\begin{proposition}\label{prop:quotient}
The homomorphism $\iota_n$ is injective and $\langle t_0', \ldots , t_{n-1}'\rangle$ is a subgroup of $\Ab$ of index two. Hence $A_{D_n}$ can be identified with the subgroup of $\Ab$ generated by the $t_i'$, $i=0, \dots, n-1$.  
\end{proposition}

\begin{proof}
We first notice that, as an immediate consequence of Lemma~\ref{lem:diag}, the subgroup $U:= \langle t_0', \ldots , t_{n-1}'\rangle\subseteq \Ab$ is proper since $W_{D_n}$ is a proper subgroup of $W_{B_n}$.


As $s'_0$ interchanges $t_0'$ and ${s'_0}t_0's_0'=t_1'$,  and as $s_0'$ commutes with $t_i'$ for $i = 2, \ldots n-1$, the involution $s_0'$ normalizes $U$ and induces on $U$ an automorphism of order $2$ 
(which is in fact an outer automorphism).
Therefore, $U  = \iota_n(A_{D_n})$ is of index $2$ in $\Ab$. 

Next we determine a presentation of $U$ using the Reidemeister-Schreier algorithm (see for instance \cite{LS}). We take as a Schreier-transversal $T:=\{1,  s_0'\}$ for the right cosets of $U$ in $\Ab$.
This yields the generating set $$\{ts_i' \overline{ts_i'}^{-1}~|~ t \in T~\mbox{and}~0 \leq i \leq n-1\} = \{t_i'~|~0 \leq i \leq n-1\}$$
where $\overline{x}$ is the representative of $Ux$ in $T$ for $x \in \Ab$. 
Application of this algorithm and of Tietze-transformations (see \cite{LS})
then precisely yields the braid relations as stated in Lemma~\ref{RelationsDn}. 
This shows that $\iota_n$ is injective.
\end{proof}

From now on we identify the subgroup $\langle t_0', t_1',\dots, t_{n-1}'\rangle\subseteq\Ab$ with $A_{D_n}$ and we set $\mathbf{t}_i = t_i'$ for all $i=0,\dots, n-1$. Note that by definition of $\Ab$, the map $q_B$ factors through $\Ab$, giving rise to a surjection $\widetilde{q}_B: \Ab\longrightarrow A_{A_{n-1}}$. Then we have 

\begin{lemma}\label{lem:comp}
The map $\iota_n$ satisfies $\widetilde{q}_B\circ \iota_n= q_D$.
\end{lemma}
\begin{proof}
We have $q_D(\mathbf{t}_0)= \sigma_1$ and $(\widetilde{q}_B\circ \iota_n)(\mathbf{t}_0)= \widetilde{q}_B(s_0' s_1' s_0')= q_B(\bs_0) q_B (\bs_1) q_B(\bs_0)= q_B(\bs_1)=\sigma_1$. For $i\geq 1$ we have $q_D(\bt_i)=\sigma_i= q_B(\bs_i)= \widetilde{q}_B(s_i')=(\widetilde{q}_B\circ \iota_n)(\bt_i)$.  
\end{proof}

\begin{definition} Given $x\in W_{D_n}$, we denote by $\mathbf{x}^D$ the canonical positive lift of $x$ in $A_{D_n}$ (which we will systematically view inside $\Ab$) and by $\mathbf{x}^B$ the canonical positive lift of $x$ in $A_{B_n}$.
\end{definition}

\begin{proposition}\label{prop:lifts}
Let $x\in W_{D_n}$. We have $\pi_{n,1}(\mathbf{x}^B)=\mathbf{x}^D$. 
\end{proposition}

\begin{proof}
Let $t_{i_1} t_{i_2}\cdots t_{i_k}$ be an $S_{D_n}$-reduced expression of $x$ in $W_{D_n}$. Replacing $t_0$ by $s_0 s_1 s_0$ and $t_i$ by $s_i$ for $i=1,\dots, n-1$ we get a word in the elements of $S_{B_n}$ for $x$. Note that this may not be a reduced expression for $x$ in $W_{B_n}$. It suffices to show that one can transform the above word into a reduced expression for $x$ in $W_{B_n}$ just by applying braid relations of type $B_n$ and the relation $s_0^2=1$.

We prove the above statement by induction on $k$. If $k=1$ then the claim holds since $t_i$, $i\geq 1$ is replaced by $s_i$ while $t_0$ is replaced by $s_0 s_1 s_0$ which is $S_{B_n}$-reduced. Hence assume that $k>1$. By induction the claim holds for $x'=t_{i_2}\cdots t_{i_k}$. By \cite[Propositions 8.1.2, 8.2.2]{BjBr} one has that $s_j$, $j\geq 1$ is a left descend of $x'$ in $W_{B_n}$ if and only if it is a left descent of $x'$ in $W_{D_n}$. Hence we can assume that $t_{i_1}=t_0$ and that it is the only left descent of $x$ in $W_{D_n}$.

Firstly, assume that $s_0$ is a left descent of $x'$ in $W_{B_n}$, hence $s_0$ is not a left descent of $s_0 x'$. We claim that it suffices to show that $s_1$ is not a left descent of $s_0 x'$: indeed, it implies that $\ell(s_0 s_1 s_0 x')=\ell(s_0 x')+2$ (where $\ell$ is the length function in $W_{B_n}$) by the lifting property (see~\cite[Corollary 2.2.8(i)]{BjBr}). Moreover by induction we can get every $S_{B_n}$-reduced decomposition of $x'$ using only the claimed relations, hence we can by induction get a reduced expression for $x'$ starting with $s_0$ with these relations. The only additional relation to apply to get a reduced decomposition of $x$ is the deletion of the $s_0^2=1$ which appears when appending $s_0 s_1 s_0$ at the left of such a reduced expression of $x'$. Hence assume that $s_1 s_0 x' < s_0 x'$ in $W_{B_n}$, i.e., that $s_1$ is a left descent of $s_0 x'$. By~\cite[Proposition 8.1.2]{BjBr} it follows that $x'^{-1} s_0 (1)>x'^{-1} s_0 (2)$ which implies that $x'^{-1}(-1) >x'^{-1}(2)$, hence $-x'^{-1}(2)>x'^{-1}(1)$. But by~\cite[Proposition 8.2.2]{BjBr} it precisely means that $t_0$ is a left descent of $x'$, a contradiction.    

Now assume that $s_0$ is not a left descent of $x'$ in $W_{B_n}$. Then $s_1$ is not a left descent of $x'$ in $W_{B_n}$, otherwise using \cite[Proposition 8.1.2]{BjBr} again it would be a left descent of $t_{i_1} x'$ in $W_{B_n}$, hence in $W_{D_n}$ by \cite[Proposition 8.2.2]{BjBr}, a contradiction. It follows that a reduced expression for $y=s_1 s_0 x'$ in $W_{B_n}$ is obtained by concatenating $s_1 s_0$ at the left of a reduced expression for $x'$ (which we can obtain by induction). If $s_0 y > y$ then we are done, while if $s_0 y < y$ then by Matsumoto's Lemma we can obtain a reduced expression of $y$ starting with $s_0$ just by applying type $B_n$ braid relations. Deleting the $s_0^2$ at the beginning of the word we then have a reduced expression of $x$.  
   
\end{proof}

\begin{rmq}

The fact that reduced expressions of an element $x\in W_{D_n}$ can be transformed into reduced expressions in $W_{B_n}$ as we did in the proof above had been noticed by Hoefsmit in his thesis~\cite[Section 2.3]{Hoef} without a proof. The fact that $A_{D_n }$ can be realized as a subgroup of $\Ab$ also implies that the corresponding Iwahori-Hecke algebra $H(W_{D_n})$ of type $D_n$ embeds into the two-parameter Iwahori-Hecke algebra $H(W_{B_n})$ of type $B_n$ where the parameter corresponding to the conjugacy class of $s_0$ is specialized at $1$. This is precisely what Hoefsmit uses to study representations of Iwahori-Hecke algebras of type $D_n$ using the representation theory of those algebras in type $B_n$.  

\end{rmq}

\section{Mikado braids of type $B_n$ and $D_n$}

\subsection{Mikado braids of type $B_n$}\label{typeB}

We recall from \cite{DG} the following

\begin{definition}
Let $(W,S)$ be a finite Coxeter system with Artin group $A_W$. An element $\beta\in A_W$ is a \defn{Mikado braid} if there exist $x,y\in W$ such that $\beta=\mathbf{x}^{-1}\mathbf{y}$. We denote by $\mathrm{Mik}(W)$ (or $\mathrm{Mik}(X)$ if $W$ is of type $X$) the set of Mikado braids in $A_W$. 
\end{definition}

We briefly recall results from \cite[Section 6.2]{DG} on topological realizations of Mikado braids in type $B_n$ which will be needed later on. The Artin group $A_{B_n}$ embeds into $A_{A_{2n-1}}$, which is isomorphic to the Artin braid group on $2n$ strands. Labeling the strands by $-n, \dots, -1, 1, \dots, n$, every simple generator in $S_{B_n}\subseteq S_{n, -n}$ is then lifted to an Artin braid as follows. The generator $\bs_0$ exchanges the strands $1$ and $-1$, while the generator $\bs_i$, $i=1,\dots, n-1$ exchanges the strands $i$ and $i+1$ as well as the strands $-i$ and $-i-1$ (in both crossings, the strand coming from the right passes over the strand coming from the left, like in the right picture in Figure~\ref{figure:ref1}). Those braids in $A_{A_{2n-1}}$ which are in $A_{B_n}$ are precisely those braids which are fixed by the automorphism which exchanges each crossing $i, i+1$ by a crossing $-i, -i-1$ of the same type, for all $i$. We call these braids \defn{symmetric}.  

There is the following graphical characterization of Mikado braids in $A_{B_n}$

\begin{theorem}[{\cite[Theorem 6.3]{DG}}]\label{thm:dg_b}
Let $\beta\in A_{B_n}$. The following are equivalent
\begin{enumerate}
\item The braid $\beta$ is a Mikado braid, that is, there are $x, y\in W_{B_n}$ such that $\beta=\mathbf{x}^{-1}\mathbf{y}$.
\item There is an Artin braid in $A_{A_{2n-1}}$ representing $\beta$, such that one can inductively remove pairs of symmetric strands, one of the two strands being above all the other strands (so that the symmetric one is under all the other strands).  
\end{enumerate}
\end{theorem} 
Note that in the second item above, we remove pairs of strands instead of single strands so that at each step of the process, the obtained braid is still symmetric (hence in $A_{B_n}$). 

\subsection{Mikado braids of type $D_n$ inside $\Ab$}

The aim of this subsection is to prove the following result, relating Mikado braids of type $D_n$ to Mikado braids of type $B_n$:

\begin{theorem}\label{thm:mikado_bd}
The Mikado braids of type $D_n$ viewed inside $\Ab$ are precisely the images of those Mikado braids of type $B_n$ which surject onto elements of $W_{D_n}$, that is, we have $$\mathrm{Mik}(D_n)=\{ \pi_{n,1}(\beta)~|~\beta\in\mathrm{Mik}(B_n)~\text{and}~\pi_n^B(\beta)\in W_{D_n}\}.$$
\end{theorem}
\begin{proof}
Let $\gamma\in\mathrm{Mik}(D_n)\subseteq\Ab$. Then there exist $x,y\in W_{D_n}$ such that $\gamma=(\mathbf{x}^D)^{-1}\mathbf{y}^D$. Note that by Lemma~\ref{lem:diag} we have $\pi_{n,2}(\gamma)=x^{-1} y\in W_{D_n}$. But by Proposition~\ref{prop:lifts} we have $\gamma=\pi_{n,1}(\beta)$ where $\beta=(\mathbf{x}^B)^{-1}\mathbf{y}^B\in\mathrm{Mik}(B_n)$, which shows the first inclusion.

Conversely, let $\beta\in \mathrm{Mik}(B_n)$ such that $\pi_{n}^B(\beta)\in W_{D_n}$. We have to show that $\pi_{n,1}(\beta)\in\mathrm{Mik}(D_n)$. By definition there are $x,y\in W_{B_n}$ such that $\beta=(\mathbf{x}^B)^{-1} \mathbf{y}^B$. Since $\pi_n^B(\beta)=x^{-1}y\in W_{D_n}$, if either $x$ or $y$ is in $W_{D_n}$ then both of them are in $W_{D_n}$ in which case we are done by Proposition~\ref{prop:lifts}. Hence assume that $x, y\notin W_{D_n}$. Since $W_{D_n}$ is a subgroup of $W_{B_n}$ of index two and $s_0\notin W_{D_n}$ there are $x', y'\in W_{D_n}$ such that $x=s_0 x'$, $y=s_0 y'$. If follows that $\mathbf{x}^B= \mathbf{s}_0^{\pm 1} \mathbf{x'}^B$ (the exponent depending on whether $s_0 x > x$ or not) and $\mathbf{y}^B=\mathbf{s}_0^{\pm 1} \mathbf{y'}^B$. Hence since the image of $\mathbf{s}_0$ in $\Ab$ has order two, using Proposition~\ref{prop:lifts} again we have $\pi_{n,1}(\beta)=(\mathbf{x'}^D)^{-1} \mathbf{y'}^D$ which concludes. 

\end{proof}

\section{Dual braid monoids}

\subsection{Noncrossing partitions}Let $(W,S)$ be a Coxeter system of spherical type. Let $T=\bigcup_{w\in W} w S w^{-1}$ denote the set of reflections in $W$ and $\lt:W\longrightarrow \mathbb{Z}_{\geq 0}$ the corresponding length function. A \defn{standard Coxeter element} in $(W,S)$ is a product of all the elements of $S$. Given $u,v\in W$, we can define a partial order $\leq_T$ on $W$ by $$u\leq_T v\Leftrightarrow \lt(u)+\lt(u^{-1}v)=\lt(v).$$ 
In this case we say that $u$ is a \defn{prefix} of $v$. 

Let $c$ be a standard Coxeter element. The set $\NC(W, c)$ of \defn{c-noncrossing partitions} consists of all the $x\in W$ such that $x\leq_T c$. The poset $(\NC(W,c),\leq_T)$ is a lattice, isomorphic to the lattice of noncrossing partitions when $W= W_{A_n} \cong \mathfrak{S}_{n+1}$. See~\cite{Arm} for more on the topic.

\begin{rmq}
There are several (unequivalent) definitions of Coxeter elements (see for instance~\cite[Section 2.2]{BGRW}). The above definitions still make sense for more general Coxeter elements, but for the realization of the dual braid monoids (which are introduced in the next section) inside Artin groups the Coxeter element is required to be standard (see~\cite[Remark 5.11]{DG}).
\end{rmq}

\subsection{Dual braid monoids}\label{Sub:DualBraid}

We recall the definition and properties of dual braid monoids. For a detailed introduction to the topic the reader is referred to \cite{Dual,DG} or \cite{Garside}. Dual braid monoids were introduced by Bessis \cite{Dual}, generalizing definitions of Birman, Ko and Lee \cite{BKL} and Bessis, Digne and Michel \cite{BDM} to all the spherical types. Let $(W, S)$ be a finite Coxeter system. Denote by $T$ the set of reflections in $W$ and by $A_W$ the corresponding Artin-Tits group. Let $c$ be a standard Coxeter element in $W$. Bessis 
defined the \defn{dual braid monoid} attached to the triple $(W, T, c)$ as follows. Take as generating set a copy  $T_c:= \{t_c~|~t \in T\}$ of $T$ and set 
$$B_c^*:=\langle t_c\in T_c~ |~t_c \in T_c, t_c t'_c= (tt't)_c t_c~\text{if }tt'\leq_T c\rangle$$ 
The defining relations of $B_c^*$ are called the \defn{dual braid relations} with respect to $c$. We mention some properties of $B_c^*$, which can be found in \cite{Dual}. The monoid $B_c^*$ is infinite and embeds into $A_W$. In fact, $B_c^*$ is a Garside monoid, hence it embeds into its group of fractions $\mathrm{Frac}(B_c^*)$ and the word problem in $\mathrm{Frac}(B_c^*)$ is solvable. Bessis showed that $\mathrm{Frac}(B_c^*)$ is isomorphic to $A_W$, but his proof requires a case-by-case analysis (see \cite[Fact 2.2.4]{Dual}) and the isomorphism is difficult to understand explicitly.  

More precisely, the embedding $B_c^*\subseteq A_W$ sends $s_c$ to $\mathbf{s}$ for every $s\in S$. In \cite[Proposition 3.13]{DG}, a formula for the elements of $T_c$ (which are the atoms of the monoid $B_c^*$) as products of the Artin generators is given, but it does not give in general a braid word of shortest possible length. 

\begin{exple}\label{ex:dual}
Let $(W,S)$ be of type $A_2$ and $c\in W$ be the Coxeter element $s_1 s_2$ where $s_i=(i,i+1)$. Then we have the dual braid relation $(s_1)_c (s_2)_c=(s_1 s_2 s_1)_c (s_1)_c$. Hence inside $A_W$, the atom $(s_1 s_2 s_1)_c$ corresponding to the non-simple reflection $s_1 s_2 s_1$ is equal to $\mathbf{s}_1\mathbf{s}_2 \mathbf{s}_1^{-1}$. 
\end{exple}

As every Garside monoid, $B_c^*$ has a finite set of \defn{simple elements}, which form a lattice under left divisibility. They are defined as follows. For $x\in \NC(W,c)$, let $x=t_1 t_2\cdots t_k$ be a \defn{$T$-reduced expression} of $x$, that is, a reduced expression as product of reflections. Then Bessis showed that the element $x_c:= (t_1)_c (t_2)_c\cdots (t_k)_c\in B_c^*$ is independent of the 
choice of the reduced expression of $x$ and therefore well-defined as a consequence of a dual Matsumoto property \cite[Section 1.6]{Dual}. The Garside element is the lift $c_c$ of $c$ and the set $\mathrm{Div}(c)$ of simple elements (that is, of (left) divisors of $c_c$) is given by $\mathrm{Div}(c):=\{ x_c~|~ x\in\NC(W,c)\}$. There is an isomorphism of posets $(\NC(W,c),\leq_T)\cong(\mathrm{Div}(c),\leq), x\mapsto x_c$, where $\leq$ is the left-divisibility order in $B_c^*$. In general, we are only able to determine the elements of $\mathrm{Div}(c)$ as words in the classical Artin generators $\mathbf{S}$ of $A_W$ by an inductive application of the dual braid relations. It is therefore difficult to study properties of elements of $\mathrm{Div}(c)$ viewed inside $A_W$. Note that the composition $B_c^*\hookrightarrow A_W\twoheadrightarrow W$ sends every product $(t_1)_c (t_2)_c\cdots (t_k)_c$, $t_i\in T$ to $t_1 t_2\cdots t_k$. 

\subsection{Standard Coxeter elements in $W_{D_n}$}\label{sec:cox}

In this subsection, we characterize standard Coxeter elements in $W_{D_n}$ in terms of signed permutations. This will be needed to introduce graphical representations of $c$-noncrossing partitions of type $D_n$ in Section~\ref{graphical}.

Recall that  $W_{D_n} \subseteq W_{B_n}$ and that $w(-i)=-w(i),~\mbox{for all}~ i\in[-n,n]$ and all $w \in W_{B_n}$.
In $W_{B_n}$, cycles of the shape $(i_1, \dots, i_r, -i_1, \dots, -i_r)$ are abbreviated by $[i_1, \dots i_r]$ and called \defn{balanced cycles}, and those
of type $(i_1, \dots, i_r)( -i_1, \dots, -i_r)$ by $((i_1, \dots, i_r ))$ and called \defn{paired cycles}. 
The set of reflections in $W_{D_n}$ is 
$$T:= T_{D_n} :=\{ (i,j)(-i,-j) \mid i, j\in\{-n, \dots, n\}, i\neq \pm j\},$$
and every $w \in W_{D_n}$ can be written as a product of disjoint cycles in which there is an even number of balanced cycles (see~\cite[Section 2]{AR}).

\begin{lemma}\label{lem:std}
An element $c\in W_{D_n}$ ($n\geq 3$) is a standard Coxeter element if and only if $c=(i_1, -i_1)(i_2, \dots, i_{n}, -i_2, \dots, -i_{n})$ where $\{ i_1, \dots, i_{n}\}=\{1, 2, 3, \dots, n\}$, $i_1\in\{1,2\}$ and the sequence $i_2\cdots i_{n}$ is first increasing, then decreasing. 
\end{lemma}
\begin{proof}
The proof is by induction on $n$. The case $n=3$ is easy to check by hand. Let $c$ be a standard Coxeter element in $W_{D_n}$, $n\geq 4$. Then either $s_n c$ or $c s_n$ is a standard Coxeter element in $W_{D_{n-1}}$, in which case induction and a straightforward computation shows that $c$ is of the required form. Conversely if $c$ is of the above form, then since $(i_1, -i_1)$ commutes with $s_n$ either $s_n c$ or $c s_n$ is of the above form in $W_{D_{n-1}}$, hence is a standard Coxeter element in $W_{D_{n-1}}$, implying that $c$ is a standard Coxeter element in $W_{D_n}$. 
\end{proof}

Elements in $\NC(W_{D_n}, c)$ will be described below via a graphical representation. 

\section{Simple dual braids of type $D_n$ are Mikado braids}\label{main}

The aim of this section is to show Theorem~\ref{theorem:main}, that is, that simple dual braids of type $D_n$ are Mikado braids.

\subsection{Outline of the proof}

The proof proceeds as follows. 
\begin{itemize}
\item\textbf{Step 1.} We describe in Section~\ref{graphical} a pictural model for the elements $x\in\NC(W_{D_n}, c)$ which is due to Athanasiadis and Reiner \cite{AR}. 
In this model the element $x$ is represented by a diagram consisting of non-intersecting polygons joining labeled points on a circle. The labeling depends on the choice of the standard Coxeter element $c$, more precisely, we first require to write the Coxeter element as a signed permutation (as in Lemma~\ref{lem:std}). 

\item\textbf{Step 2.} We slightly modify the diagram from Step $1$ associated to $x\in\NC(W_{D_n}, c)$ to obtain a new diagram $N_x$ consisting of non-intersecting polygons joining labeled points on a circle. The only difference with the Athanasiadis-Reiner model is that there is a point with two labels in the latter, which we split in two different points. As we will see, the diagram $N_x$ is not unique in general, but we will show that all the information which we will use from the diagram $N_x$ is independent of the chosen diagram representing $x$. From this new diagram $N_x$, we build a topological braid $\beta_x$ lying in an Artin group $A_{B_n}$ of type $B_n$ (viewed inside $A_{A_{2n-1}}$, hence $\beta_x$ is a symmetric braid on $2n$ strands). We first explain how to define the diagram $N_x$ for elements of $T_{D_n}\subseteq\NC(W_{D_n}, c)$ and we then do it for all $x\in\NC(W_{D_n}, c)$. 

\item\textbf{Step 3.} We show that the braids $\pi_{n, 1}(\beta_t)\in \Ab$, for $t\in T_{D_n}$,  lie in $A_{D_n}$ and satisfy the dual braid relations with respect to $c$. This will follow  from the more general statement that if $x \leq_T xt\leq_T c$ with $t\in T_{D_n}$, then $\pi_{n, 1}(\beta_x) \pi_{n, 1}(\beta_t)=\pi_{n,1}(\beta_{xt})$. This property and the fact that $\pi_{n, 1}(\beta_s)=\mathbf{s}$ for all $s\in S_{D_n}$ will be enough to conclude that $\pi_{n, 1}(\beta_x)$ is equal to the simple dual braid $x_c$ for all $x\in \NC(W_{D_n}, c)$ (this is explained in the proof of Corollary~\ref{cor:sdb}). In particular we also show that $\pi_{n,1}(\beta_x)$ does not depend on the choice of the diagram $N_x$. 


\item\textbf{Step 4.} We show that the braid $\beta_x$, $x\in \NC(W_{D_n}, c)$ is a Mikado braid in $A_{B_n}$ by using the topological characterization of \cite{DG}. Recall that $\beta_x$ is defined graphically, as an Artin braid on $2n$ strands. Together with Step $3$ and Theorem~\ref{thm:mikado_bd}, it follows that $x_c=\pi_{n, 1}(\beta_x)$ is a Mikado braid, which proves Theorem~\ref{theorem:main}. 

\end{itemize}



\subsection{Graphical model for noncrossing partitions}\label{graphical}

Athanasiadis and Reiner found a graphical model for noncrossing partitions of type $D_n$. We present it here (with slightly different conventions). First we explain how to label a circle depending on the choice of the standard Coxeter element $c$. 

Given a standard Coxeter element $c=(i_1, -i_1)(i_2, \dots, i_{n}, -i_2, \dots, -i_{n})$ in $W_{D_n}$, where the  notation is as in Lemma~\ref{lem:std} and where $i_2=-n$, we place $2n-2$ points (labeled by $i_2, \dots, i_n, -i_2, \dots, -i_n$) on a circle as follows: point $-n$ is at the top of the circle while point $n$ is at the bottom. The remaining points all have distinct height depending on their label: if $i<j$ then point $i$ is higher than point $j$. Moreover, when going along the circle in clockwise order starting at $i_2=-n$, the points must be met in the order $i_2 i_3 \cdots i_n (-i_2) (-i_3)\cdots (-i_n)$. Finally, we add a point at the center of the circle, labeled by $\pm i_1$. 


Athanasiadis and Reiner showed that $c$-noncrossing partitions are those for which there exists a graphical representation as follows (in their description, we have $i_1= n$; this corresponds to a choice of Coxeter element which is not standard, however by conjugation we can assume it to be standard an to have $i_1\in\{ 1, 2\}$. The $c$-noncrossing partition lattices are isomorphic for all Coxeter elements $c$). Given  $x\in\NC(W_{D_n}, c)$, consider its cycle decomposition inside $S_{-n,n}$ and associate to each cycle the polygon given by the convex hull of the points labeled by elements in the support of the cycle. It results in a noncrossing diagram, i.e., the various obtained polygons do not intersect, with two possible exceptions: if there is a polygon $Q$ of $x$ with $i_1\in Q$, $-i_1\notin Q$, then $-Q$ is also a polygon of $x$. Thus the two polygons $Q$ and $-Q$ will have the middle point in common (Note that since $x$ is a signed permutation, for every polygon $P$ of $x$ we have that $-P$ is also a polygon of $x$, possibly with $P=-P$). The second case appears when the decomposition of $x$ has a product of factors of the form $[j][i_1]$ for some $j\neq \pm i_1$. In this case to avoid confusion with the noncrossing representation of the reflection $((j, i_1))$ (or $((j, -i_1))$) we have to choose an alternative way of representing this product. Note that the cycle $[j]$ should be considered as a polygon $P$ such that $P=-P$. By analogy with the situation where there is such a polygon and where the point $\pm i_1$ lies inside $P$, we represent $[j]$ by two curves both joining $j$ to $-j$ and not intersecting except at the points $\pm j$, in such a way that the point $\pm i_1$ lies between these two curves. 

Conversely, to every noncrossing diagram with the above properties, one can associate an element $x$ of $\NC(W_{D_n},c)$ as follows: we send each polygon $P$ with labels $j_1, j_2, \dots, j_k$ (read in clockwise order) to the cycle $(j_1, j_2, \dots, j_k)$ except in case $P=-P$. Each single point with label $i$ is sent to the one-cycle $(i)$ except $i_1$ in case there is a polygon $P$ with $P=-P$ (in which case $\pm i_1$ lie inside $P$). In this last case, if $P$ is labeled by $j_1,j_2,\dots,j_k$ then we send it to the product of cycles $(i_1, -i_1)(j_1,j_2,\dots,j_k)$ (like in the middle example of Figure~\ref{figure:ref4}). The element $x$ is then the product of all the cycles associated to all the polygons of the noncrossing diagram (note that they are disjoint). Note that when the middle point lies in two different polygons, one has to specify in which polygon the label $i_1$ lies. Examples are given in Figure~\ref{figure:ref4} and we refer to~\cite{AR} for more details. 




\begin{figure}[h!]

\begin{tabular}{ccc}

\begin{pspicture}(-2,-1.92)(2.5,1.92)
\pscircle[linecolor=gray, linewidth=0.2pt](0,0){1.8}

\psline[linecolor=red](0.897,1.56)(1.44,1.08)(0,0)(0.897,1.56)
\psline[linecolor=red](-0.897,-1.56)(-1.44,-1.08)(0,0)(-0.897,-1.56)
\psline[linecolor=red](0,1.8)(-1.796,-0.12)
\psline[linecolor=red](0,-1.8)(1.796,0.12)

\psdots(0,1.8)(0.897,1.56)(-1.2237,1.32)(1.44,1.08)(-1.59197,0.84)(1.697,0.6)(0,0)(1.796,0.12)(-1.796,-0.12)(-1.697,-0.6)(1.59197,-0.84)(-1.44,-1.08)(1.2237,-1.32)(-0.897,-1.56)(0,-1.8)


\psline[linestyle=dotted, linewidth=0.4pt](-2,1.8)(2,1.8)
\psline[linestyle=dotted, linewidth=0.4pt](-2,1.56)(2,1.56)
\psline[linestyle=dotted, linewidth=0.4pt](-2,1.32)(2,1.32)
\psline[linestyle=dotted, linewidth=0.4pt](-2,1.08)(2,1.08)
\psline[linestyle=dotted, linewidth=0.4pt](-2,0.84)(2,0.84)
\psline[linestyle=dotted, linewidth=0.4pt](-2,0.6)(2,0.6)
\psline[linestyle=dotted, linewidth=0.4pt](-2,0.36)(2,0.36)
\psline[linestyle=dotted, linewidth=0.4pt](-2,0.12)(2,0.12)
\psline[linestyle=dotted, linewidth=0.4pt](-2,-1.8)(2,-1.8)
\psline[linestyle=dotted, linewidth=0.4pt](-2,-1.56)(2,-1.56)
\psline[linestyle=dotted, linewidth=0.4pt](-2,-1.32)(2,-1.32)
\psline[linestyle=dotted, linewidth=0.4pt](-2,-1.08)(2,-1.08)
\psline[linestyle=dotted, linewidth=0.4pt](-2,-0.84)(2,-0.84)
\psline[linestyle=dotted, linewidth=0.4pt](-2,-0.6)(2,-0.6)
\psline[linestyle=dotted, linewidth=0.4pt](-2,-0.36)(2,-0.36)
\psline[linestyle=dotted, linewidth=0.4pt](-2,-0.12)(2,-0.12)

\rput(2.2, 1.8){\tiny $-8$}
\rput(2.2, 1.56){\tiny $-7$}
\rput(2.2, 1.32){\tiny $-6$}
\rput(2.2, 1.08){\tiny $-5$}
\rput(2.2, 0.84){\tiny $-4$}
\rput(2.2, 0.6){\tiny $-3$}
\rput(-0.4,0){\tiny $-2$}
\rput(0.4,0){\tiny $2$}
\rput(2.2, 0.12){\tiny $-1$}
\rput(2.2, -0.12){\tiny $1$}
\rput(2.2, -0.6){\tiny $3$}
\rput(2.2, -0.84){\tiny $4$}
\rput(2.2, -1.08){\tiny $5$}
\rput(2.2, -1.32){\tiny $6$}
\rput(2.2, -1.56){\tiny $7$}
\rput(2.2, -1.8){\tiny $8$}
\end{pspicture}

&

\begin{pspicture}(-2,-1.92)(2.5,1.92)
\pscircle[linecolor=gray, linewidth=0.2pt](0,0){1.8}

\psline[linecolor=red](-1.2237,1.32)(1.697,0.6)(1.796,0.12)(1.2237,-1.32)(-1.697,-0.6)(-1.796,-0.12)(-1.2237,1.32)
\psline[linecolor=red](0,1.8)(0.897,1.56)(1.44,1.08)(0,1.8)
\psline[linecolor=red](0,-1.8)(-0.897,-1.56)(-1.44,-1.08)(0,-1.8)

\psdots(0,1.8)(0.897,1.56)(-1.2237,1.32)(1.44,1.08)(-1.59197,0.84)(1.697,0.6)(0,0)(1.796,0.12)(-1.796,-0.12)(-1.697,-0.6)(1.59197,-0.84)(-1.44,-1.08)(1.2237,-1.32)(-0.897,-1.56)(0,-1.8)


\psline[linestyle=dotted, linewidth=0.4pt](-2,1.8)(2,1.8)
\psline[linestyle=dotted, linewidth=0.4pt](-2,1.56)(2,1.56)
\psline[linestyle=dotted, linewidth=0.4pt](-2,1.32)(2,1.32)
\psline[linestyle=dotted, linewidth=0.4pt](-2,1.08)(2,1.08)
\psline[linestyle=dotted, linewidth=0.4pt](-2,0.84)(2,0.84)
\psline[linestyle=dotted, linewidth=0.4pt](-2,0.6)(2,0.6)
\psline[linestyle=dotted, linewidth=0.4pt](-2,0.36)(2,0.36)
\psline[linestyle=dotted, linewidth=0.4pt](-2,0.12)(2,0.12)
\psline[linestyle=dotted, linewidth=0.4pt](-2,-1.8)(2,-1.8)
\psline[linestyle=dotted, linewidth=0.4pt](-2,-1.56)(2,-1.56)
\psline[linestyle=dotted, linewidth=0.4pt](-2,-1.32)(2,-1.32)
\psline[linestyle=dotted, linewidth=0.4pt](-2,-1.08)(2,-1.08)
\psline[linestyle=dotted, linewidth=0.4pt](-2,-0.84)(2,-0.84)
\psline[linestyle=dotted, linewidth=0.4pt](-2,-0.6)(2,-0.6)
\psline[linestyle=dotted, linewidth=0.4pt](-2,-0.36)(2,-0.36)
\psline[linestyle=dotted, linewidth=0.4pt](-2,-0.12)(2,-0.12)

\rput(2.2, 1.8){\tiny $-8$}
\rput(2.2, 1.56){\tiny $-7$}
\rput(2.2, 1.32){\tiny $-6$}
\rput(2.2, 1.08){\tiny $-5$}
\rput(2.2, 0.84){\tiny $-4$}
\rput(2.2, 0.6){\tiny $-3$}
\rput(-0.4,0){\tiny $-2$}
\rput(0.4,0){\tiny $2$}
\rput(2.2, 0.12){\tiny $-1$}
\rput(2.2, -0.12){\tiny $1$}
\rput(2.2, -0.6){\tiny $3$}
\rput(2.2, -0.84){\tiny $4$}
\rput(2.2, -1.08){\tiny $5$}
\rput(2.2, -1.32){\tiny $6$}
\rput(2.2, -1.56){\tiny $7$}
\rput(2.2, -1.8){\tiny $8$}
\end{pspicture}

&

\begin{pspicture}(-2,-1.92)(2.5,1.92)
\pscircle[linecolor=gray, linewidth=0.2pt](0,0){1.8}

\psline[linecolor=red](-1.59197,0.84)(1.2237,-1.32)(-1.697,-0.6)(-1.59197,0.84)
\psline[linecolor=red](1.59197,-0.84)(-1.2237,1.32)(1.697,0.6)(1.59197,-0.84)

\psdots(0,1.8)(0.897,1.56)(-1.2237,1.32)(1.44,1.08)(-1.59197,0.84)(1.697,0.6)(0,0)(1.796,0.12)(-1.796,-0.12)(-1.697,-0.6)(1.59197,-0.84)(-1.44,-1.08)(1.2237,-1.32)(-0.897,-1.56)(0,-1.8)


\psline[linestyle=dotted, linewidth=0.4pt](-2,1.8)(2,1.8)
\psline[linestyle=dotted, linewidth=0.4pt](-2,1.56)(2,1.56)
\psline[linestyle=dotted, linewidth=0.4pt](-2,1.32)(2,1.32)
\psline[linestyle=dotted, linewidth=0.4pt](-2,1.08)(2,1.08)
\psline[linestyle=dotted, linewidth=0.4pt](-2,0.84)(2,0.84)
\psline[linestyle=dotted, linewidth=0.4pt](-2,0.6)(2,0.6)
\psline[linestyle=dotted, linewidth=0.4pt](-2,0.36)(2,0.36)
\psline[linestyle=dotted, linewidth=0.4pt](-2,0.12)(2,0.12)
\psline[linestyle=dotted, linewidth=0.4pt](-2,-1.8)(2,-1.8)
\psline[linestyle=dotted, linewidth=0.4pt](-2,-1.56)(2,-1.56)
\psline[linestyle=dotted, linewidth=0.4pt](-2,-1.32)(2,-1.32)
\psline[linestyle=dotted, linewidth=0.4pt](-2,-1.08)(2,-1.08)
\psline[linestyle=dotted, linewidth=0.4pt](-2,-0.84)(2,-0.84)
\psline[linestyle=dotted, linewidth=0.4pt](-2,-0.6)(2,-0.6)
\psline[linestyle=dotted, linewidth=0.4pt](-2,-0.36)(2,-0.36)
\psline[linestyle=dotted, linewidth=0.4pt](-2,-0.12)(2,-0.12)

\rput(2.2, 1.8){\tiny $-8$}
\rput(2.2, 1.56){\tiny $-7$}
\rput(2.2, 1.32){\tiny $-6$}
\rput(2.2, 1.08){\tiny $-5$}
\rput(2.2, 0.84){\tiny $-4$}
\rput(2.2, 0.6){\tiny $-3$}
\rput(-0.4,0){\tiny $-2$}
\rput(0.4,0){\tiny $2$}
\rput(2.2, 0.12){\tiny $-1$}
\rput(2.2, -0.12){\tiny $1$}
\rput(2.2, -0.6){\tiny $3$}
\rput(2.2, -0.84){\tiny $4$}
\rput(2.2, -1.08){\tiny $5$}
\rput(2.2, -1.32){\tiny $6$}
\rput(2.2, -1.56){\tiny $7$}
\rput(2.2, -1.8){\tiny $8$}
\end{pspicture}

\end{tabular}

\caption{Examples of noncrossing diagrams for $x_1=((1,-8))((7,5,-2))$, $x_2=((8,7,5))[6,3,1][2]$, $x_3=((6,3,-4))\in\NC(W_{D_n},c)$.}
\label{figure:ref4}

\end{figure}
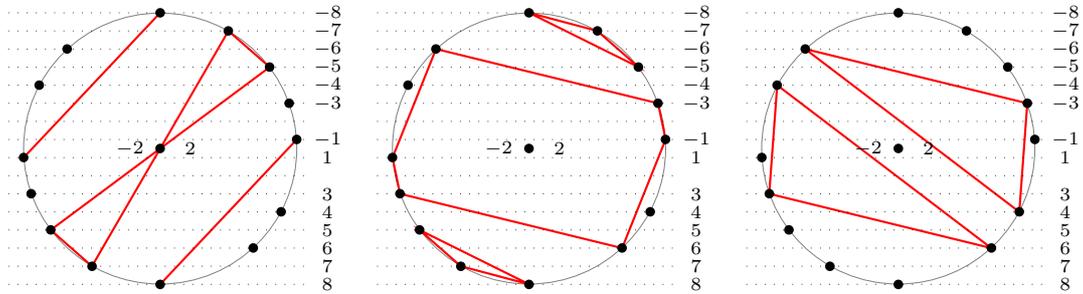

\subsection{The diagram $N_x$ and the braid $\beta_x$}

To define the diagram $N_x$, we slightly modify the labeling of the circle given in the previous section by splitting the point $\pm i_1$ into two points placed on the vertical axis of the circle consistently with their labels (all the points should be placed such that the point $i$ is higher than the point $j$ if $i<j$). An example is given in Figure~\ref{figure:ordre} and we call this labeling the \defn{$c$-labeling} of the circle. The idea is then to start from Athanasiadis and Reiner's graphical representation of $x\in\NC(W_{D_n}, c)$ and just split the middle point into two points. For convenience we may represent the polygons by curvilinear polygons since in some cases, because of the splitting it might not be possible to have the polygons not intersecting each other. Depending on the situation we will add an edge joining the two points $i_1$ and $-i_1$:  we explain more in details below how to draw the diagrams $N_x$, first when $x$ is a reflection, then in general. 

\begin{figure}[h!]
\centering
{\begin{minipage}{7cm}
\psscalebox{1.3}{\begin{pspicture}(-2,-1.92)(2.5,1.92)
\pscircle[linecolor=gray, linewidth=0.2pt](0,0){1.8}

\psdots(0,1.8)(0.897,1.56)(-1.2237,1.32)(1.44,1.08)(-1.59197,0.84)(1.697,0.6)(0,0.36)(1.796,0.12)(-1.796,-0.12)(0,-0.36)(-1.697,-0.6)(1.59197,-0.84)(-1.44,-1.08)(1.2237,-1.32)(-0.897,-1.56)(0,-1.8)


\psline[linestyle=dotted, linewidth=0.4pt](-2,1.8)(2,1.8)
\psline[linestyle=dotted, linewidth=0.4pt](-2,1.56)(2,1.56)
\psline[linestyle=dotted, linewidth=0.4pt](-2,1.32)(2,1.32)
\psline[linestyle=dotted, linewidth=0.4pt](-2,1.08)(2,1.08)
\psline[linestyle=dotted, linewidth=0.4pt](-2,0.84)(2,0.84)
\psline[linestyle=dotted, linewidth=0.4pt](-2,0.6)(2,0.6)
\psline[linestyle=dotted, linewidth=0.4pt](-2,0.36)(2,0.36)
\psline[linestyle=dotted, linewidth=0.4pt](-2,0.12)(2,0.12)
\psline[linestyle=dotted, linewidth=0.4pt](-2,-1.8)(2,-1.8)
\psline[linestyle=dotted, linewidth=0.4pt](-2,-1.56)(2,-1.56)
\psline[linestyle=dotted, linewidth=0.4pt](-2,-1.32)(2,-1.32)
\psline[linestyle=dotted, linewidth=0.4pt](-2,-1.08)(2,-1.08)
\psline[linestyle=dotted, linewidth=0.4pt](-2,-0.84)(2,-0.84)
\psline[linestyle=dotted, linewidth=0.4pt](-2,-0.6)(2,-0.6)
\psline[linestyle=dotted, linewidth=0.4pt](-2,-0.36)(2,-0.36)
\psline[linestyle=dotted, linewidth=0.4pt](-2,-0.12)(2,-0.12)

\rput(2.2, 1.8){\tiny $-8$}
\rput(2.2, 1.56){\tiny $-7$}
\rput(2.2, 1.32){\tiny $-6$}
\rput(2.2, 1.08){\tiny $-5$}
\rput(2.2, 0.84){\tiny $-4$}
\rput(2.2, 0.6){\tiny $-3$}
\rput(2.2, 0.36){\tiny $-2$}
\rput(2.2, 0.12){\tiny $-1$}
\rput(2.2, -0.12){\tiny $1$}
\rput(2.2, -0.36){\tiny $2$}
\rput(2.2, -0.6){\tiny $3$}
\rput(2.2, -0.84){\tiny $4$}
\rput(2.2, -1.08){\tiny $5$}
\rput(2.2, -1.32){\tiny $6$}
\rput(2.2, -1.56){\tiny $7$}
\rput(2.2, -1.8){\tiny $8$}
\end{pspicture}}
 \end{minipage}
\begin{minipage}{6.5cm} 
\caption{}
\label{figure:ordre}
Example of a $c$-labeling in type $D_8$. Here $c=t_1 t_3 t_5 t_7 t_6 t_4 t_2 t_0=(2, -2)[-8,-7,-5,-3,-1,4,6]$ and $i_1=2$. \end{minipage}}

\end{figure}



\subsubsection{Pictures for reflections}\label{pic:ref}

Reflections are all of the form $t= c_1 c_2$, where $c_1$ and $c_2$ are two $2$-cycles with opposite support. If $c_1=(i,j)$, we will draw a curvilinear ``polygon'' with two edges both joining $i$ to $j$. We then orient the polygon in counterclockwise order. We do the same for $c_2=(-i,-j)$ in such a way that the second curvilinear polygon does not intersect the first one. 
In some cases, there is not a unique way of drawing two such curvilinear polygons with the condition that the resulting diagram should be noncrossing. We explain how to do it in the next paragraph by separating the set of reflections into three classes. 

 Firstly, assume that $\supp(c_1)=\{i, j\}\subseteq \{1, \dots, n\}$, then $N_t$ is drawn as in the left picture of Figure~\ref{figure:ref1}. Now assume that $\supp(c_1)=\{i, -j\}$ with $i\in \{1,\dots, n\}\backslash \{i_1\}$, $j\in\{-1,\dots, -n\}\backslash \{-i_1\}$. In that case, we draw the two curvilinear polygons in such a way that the two middle points labeled by $\pm i_1$ lie between them, as done in Figure~\ref{figure:ref2}. The last case is the case where $c_1=(i_1, j)$ with $j\in\{-1,\dots,-n\}\backslash\{-i_1\}$. In that case, there are two ways of drawing the curvilinear polygon (see the left pictures of Figure~\ref{figure:ref3}). We can choose any of the two pictures for $N_t$. 


Starting from such a noncrossing diagram, we then associate an Artin braid $\beta_t$ on $2n$ strands to it, by first projecting the noncrossing diagram to the right (as done in the left pictures of Figures~\ref{figure:ref1} and \ref{figure:ref2}), i.e., putting all the points on the same vertical line, obtaining a new graph for the noncrossing partition. This new graph can then be viewed as a braid diagram, viewed from the bottom: a curve joining point $k$ to point $\ell$ corresponds to a $k$-th strand ending at $\ell$, while single points without a curve starting or ending at them correspond to unbraided strands. If a point has nothing at its right (resp. at its left), it means that the corresponding unbraided strand is above all the others (resp. below all the others). The points lying right to (resp. left to) a curve correspond to an unbraided strand lying above (resp. below) the strand corresponding to that curve. See the above mentioned Figures. Note that in the case of Figure~\ref{figure:ref3}, the two braids $\beta_t$ obtained from the two different diagrams $N_t$ are distinct in $A_{B_n}$, but their images $\pi_{n, 1}(\beta_t)$ in the quotient $\Ab$ are the same because we can invert the crossings corresponding to the generator $\mathbf{s}_0$ (because of the relation $\mathbf{s}_0^2=1$ which holds in the quotient).

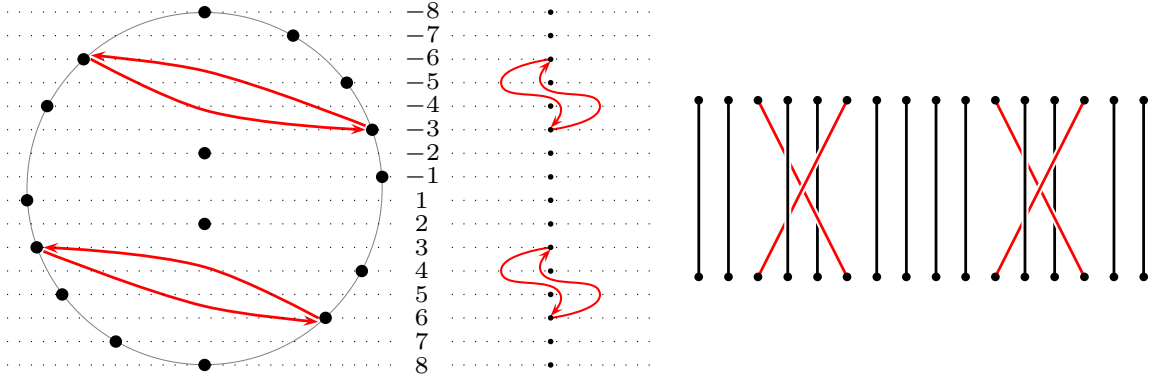
\begin{figure}[h!]

\psscalebox{1.3}{\begin{pspicture}(-2,-1.92)(10,1.92)
\pscircle[linecolor=gray, linewidth=0.2pt](0,0){1.8}

\pscurve[linecolor=red, linewidth=0.6pt]{->}(3.5,0.6)(4,0.84)(3.4,1.08)(3.5,1.3)
\pscurve[linecolor=red, linewidth=0.6pt]{<-}(3.5,0.62)(3.6,0.84)(3,1.08)(3.5,1.32)

\pscurve[linecolor=red, linewidth=0.6pt]{->}(3.5,-0.6)(3,-0.84)(3.6,-1.08)(3.5,-1.3)
\pscurve[linecolor=red, linewidth=0.6pt]{<-}(3.5,-0.62)(3.4,-0.84)(4,-1.08)(3.5,-1.32)

\pscurve[linecolor=red]{<-}(-1.15,1.36)(0,1.2)(1.63,0.64)
\pscurve[linecolor=red]{<-}(1.15,-1.36)(0,-1.2)(-1.63,-0.64)

\pscurve[linecolor=red]{->}(-1.15,1.32)(0,0.8)(1.63,0.6)
\pscurve[linecolor=red]{->}(1.15,-1.32)(0,-0.8)(-1.63,-0.6)

\psdots(0,1.8)(0.897,1.56)(-1.2237,1.32)(1.44,1.08)(-1.59197,0.84)(1.697,0.6)(0,0.36)(1.796,0.12)(-1.796,-0.12)(0,-0.36)(-1.697,-0.6)(1.59197,-0.84)(-1.44,-1.08)(1.2237,-1.32)(-0.897,-1.56)(0,-1.8)

\psdots[dotsize=1.6pt](3.5,1.8)(3.5,1.56)(3.5,1.32)(3.5,1.08)(3.5,0.84)(3.5,0.6)(3.5,0.36)(3.5,0.12)(3.5,-0.12)(3.5,-0.36)(3.5,-0.6)(3.5,-0.84)(3.5,-1.08)(3.5,-1.32)(3.5,-1.56)(3.5,-1.8)


\psline[linestyle=dotted, linewidth=0.4pt](-2,1.8)(2,1.8)
\psline[linestyle=dotted, linewidth=0.4pt](-2,1.56)(2,1.56)
\psline[linestyle=dotted, linewidth=0.4pt](-2,1.32)(2,1.32)
\psline[linestyle=dotted, linewidth=0.4pt](-2,1.08)(2,1.08)
\psline[linestyle=dotted, linewidth=0.4pt](-2,0.84)(2,0.84)
\psline[linestyle=dotted, linewidth=0.4pt](-2,0.6)(2,0.6)
\psline[linestyle=dotted, linewidth=0.4pt](-2,0.36)(2,0.36)
\psline[linestyle=dotted, linewidth=0.4pt](-2,0.12)(2,0.12)
\psline[linestyle=dotted, linewidth=0.4pt](-2,-1.8)(2,-1.8)
\psline[linestyle=dotted, linewidth=0.4pt](-2,-1.56)(2,-1.56)
\psline[linestyle=dotted, linewidth=0.4pt](-2,-1.32)(2,-1.32)
\psline[linestyle=dotted, linewidth=0.4pt](-2,-1.08)(2,-1.08)
\psline[linestyle=dotted, linewidth=0.4pt](-2,-0.84)(2,-0.84)
\psline[linestyle=dotted, linewidth=0.4pt](-2,-0.6)(2,-0.6)
\psline[linestyle=dotted, linewidth=0.4pt](-2,-0.36)(2,-0.36)
\psline[linestyle=dotted, linewidth=0.4pt](-2,-0.12)(2,-0.12)

\psline[linestyle=dotted, linewidth=0.4pt](2.5,1.8)(4.5,1.8)
\psline[linestyle=dotted, linewidth=0.4pt](2.5,1.56)(4.5,1.56)
\psline[linestyle=dotted, linewidth=0.4pt](2.5,1.32)(4.5,1.32)
\psline[linestyle=dotted, linewidth=0.4pt](2.5,1.08)(4.5,1.08)
\psline[linestyle=dotted, linewidth=0.4pt](2.5,0.84)(4.5,0.84)
\psline[linestyle=dotted, linewidth=0.4pt](2.5,0.6)(4.5,0.6)
\psline[linestyle=dotted, linewidth=0.4pt](2.5,0.36)(4.5,0.36)
\psline[linestyle=dotted, linewidth=0.4pt](2.5,0.12)(4.5,0.12)
\psline[linestyle=dotted, linewidth=0.4pt](2.5,-1.8)(4.5,-1.8)
\psline[linestyle=dotted, linewidth=0.4pt](2.5,-1.56)(4.5,-1.56)
\psline[linestyle=dotted, linewidth=0.4pt](2.5,-1.32)(4.5,-1.32)
\psline[linestyle=dotted, linewidth=0.4pt](2.5,-1.08)(4.5,-1.08)
\psline[linestyle=dotted, linewidth=0.4pt](2.5,-0.84)(4.5,-0.84)
\psline[linestyle=dotted, linewidth=0.4pt](2.5,-0.6)(4.5,-0.6)
\psline[linestyle=dotted, linewidth=0.4pt](2.5,-0.36)(4.5,-0.36)
\psline[linestyle=dotted, linewidth=0.4pt](2.5,-0.12)(4.5,-0.12)

\rput(2.2, 1.8){\tiny $-8$}
\rput(2.2, 1.56){\tiny $-7$}
\rput(2.2, 1.32){\tiny $-6$}
\rput(2.2, 1.08){\tiny $-5$}
\rput(2.2, 0.84){\tiny $-4$}
\rput(2.2, 0.6){\tiny $-3$}
\rput(2.2, 0.36){\tiny $-2$}
\rput(2.2, 0.12){\tiny $-1$}
\rput(2.2, -0.12){\tiny $1$}
\rput(2.2, -0.36){\tiny $2$}
\rput(2.2, -0.6){\tiny $3$}
\rput(2.2, -0.84){\tiny $4$}
\rput(2.2, -1.08){\tiny $5$}
\rput(2.2, -1.32){\tiny $6$}
\rput(2.2, -1.56){\tiny $7$}
\rput(2.2, -1.8){\tiny $8$}

\psline(5,-0.9)(5,0.9)
\psline(5.3,-0.9)(5.3,0.9)
\psline(6.2,-0.9)(6.2,0.9)
\psline(6.8,-0.9)(6.8,0.9)
\psline(7.1,-0.9)(7.1,0.9)
\psline(7.4,-0.9)(7.4,0.9)
\psline(7.7,-0.9)(7.7,0.9)
\psline(8.6,-0.9)(8.6,0.9)
\psline(9.2,-0.9)(9.2,0.9)
\psline(9.5,-0.9)(9.5,0.9)
\psline[linecolor=white, linewidth=2pt](6.5, -0.9)(5.6,0.9)
\psline[linecolor=red](6.5, -0.9)(5.6,0.9)

\psline[linecolor=white, linewidth=2pt](6.5, 0.9)(5.6,-0.9)
\psline[linecolor=red](6.5, 0.9)(5.6,-0.9)

\psline[linecolor=white, linewidth=2pt](8.9,- 0.9)(8,0.9)
\psline[linecolor=red](8.9, -0.9)(8,0.9)
\psline[linecolor=white, linewidth=2pt](8.9, 0.9)(8,-0.9)
\psline[linecolor=red](8.9, 0.9)(8,-0.9)

\psline[linecolor=white, linewidth=2pt](5.9,0.9)(5.9,-0.9)
\psline(5.9,0.9)(5.9,-0.9)
\psline[linecolor=white, linewidth=2pt](8.3,0.9)(8.3,-0.9)
\psline(8.3,0.9)(8.3,-0.9)

\psdots[dotsize=2.5pt](5,0.9)(5.3,0.9)(5.6,0.9)(5.9,0.9)(6.2,0.9)(6.5,0.9)(6.8,0.9)(7.1,0.9)(7.4,0.9)(7.7,0.9)(8,0.9)(8.3,0.9)(8.6,0.9)(8.9,0.9)(9.2,0.9)(9.5,0.9)

\psdots[dotsize=2.5pt](5,-0.9)(5.3,-0.9)(5.6,-0.9)(5.9,-0.9)(6.2,-0.9)(6.5,-0.9)(6.8,-0.9)(7.1,-0.9)(7.4,-0.9)(7.7,-0.9)(8,-0.9)(8.3,-0.9)(8.6,-0.9)(8.9,-0.9)(9.2,-0.9)(9.5,-0.9)

\end{pspicture}}

\caption{The diagram $N_t$ for $t=(3,6)(-3,-6)$ and the braid $\beta_t$.}
\label{figure:ref1}

\end{figure}

\begin{figure}[h!]

\centering
\psscalebox{1.3}{\begin{pspicture}(-2,-1.92)(10,1.92)
\pscircle[linecolor=gray, linewidth=0.2pt](0,0){1.8}

\pscurve[linecolor=red, linewidth=0.6pt]{<-}(3.52,1.08)(3.9,0.84)(3.2,0.66)(3,0.46)(3,0.26)(3.2,0)(3.7,-0.12)(3.35,-0.36)(3.5,-0.6)
\pscurve[linecolor=red, linewidth=0.6pt]{->}(3.5,1.08)(3.65,0.89)(2.9,0.65)(2.8,0.42)(3.1,-0.01)(3.25,-0.06)(3.35,-0.07)(3.6,-0.12)(3.2,-0.36)(3.48,-0.6)

\pscurve[linecolor=red, linewidth=0.6pt]{<-}(3.48,-1.08)(3.1,-0.84)(3.8,-0.66)(4,-0.46)(4,-0.26)(3.8,0)(3.3,0.12)(3.65,0.36)(3.5,0.6)
\pscurve[linecolor=red, linewidth=0.6pt]{->}(3.5,-1.08)(3.35,-0.89)(4.1,-0.65)(4.2,-0.42)(3.9,0.01)(3.75,0.06)(3.65,0.07)(3.4,0.12)(3.8,0.36)(3.52,0.6)

\pscurve[linecolor=red]{->}(1.44,1.08)(1.3,1.13)(0,0.9)(-1.62,-0.4)(-1.69,-0.56)
\pscurve[linecolor=red]{->}(-1.44,-1.08)(-1.3,-1.13)(0,-0.9)(1.62,0.4)(1.69,0.56)

\pscurve[linecolor=red]{<-}(1.395,1.08)(0,0.6)(-1.69,-0.58)
\pscurve[linecolor=red]{<-}(-1.395,-1.08)(0,-0.6)(1.69,0.58)

\psdots(0,1.8)(0.897,1.56)(-1.2237,1.32)(1.44,1.08)(-1.59197,0.84)(1.697,0.6)(0,0.36)(1.796,0.12)(-1.796,-0.12)(0,-0.36)(-1.697,-0.6)(1.59197,-0.84)(-1.44,-1.08)(1.2237,-1.32)(-0.897,-1.56)(0,-1.8)

\psdots[dotsize=1.6pt](3.5,1.8)(3.5,1.56)(3.5,1.32)(3.5,1.08)(3.5,0.84)(3.5,0.6)(3.5,0.36)(3.5,0.12)(3.5,-0.12)(3.5,-0.36)(3.5,-0.6)(3.5,-0.84)(3.5,-1.08)(3.5,-1.32)(3.5,-1.56)(3.5,-1.8)

\psline[linestyle=dotted, linewidth=0.4pt](-2,1.8)(2,1.8)
\psline[linestyle=dotted, linewidth=0.4pt](-2,1.56)(2,1.56)
\psline[linestyle=dotted, linewidth=0.4pt](-2,1.32)(2,1.32)
\psline[linestyle=dotted, linewidth=0.4pt](-2,1.08)(2,1.08)
\psline[linestyle=dotted, linewidth=0.4pt](-2,0.84)(2,0.84)
\psline[linestyle=dotted, linewidth=0.4pt](-2,0.6)(2,0.6)
\psline[linestyle=dotted, linewidth=0.4pt](-2,0.36)(2,0.36)
\psline[linestyle=dotted, linewidth=0.4pt](-2,0.12)(2,0.12)
\psline[linestyle=dotted, linewidth=0.4pt](-2,-1.8)(2,-1.8)
\psline[linestyle=dotted, linewidth=0.4pt](-2,-1.56)(2,-1.56)
\psline[linestyle=dotted, linewidth=0.4pt](-2,-1.32)(2,-1.32)
\psline[linestyle=dotted, linewidth=0.4pt](-2,-1.08)(2,-1.08)
\psline[linestyle=dotted, linewidth=0.4pt](-2,-0.84)(2,-0.84)
\psline[linestyle=dotted, linewidth=0.4pt](-2,-0.6)(2,-0.6)
\psline[linestyle=dotted, linewidth=0.4pt](-2,-0.36)(2,-0.36)
\psline[linestyle=dotted, linewidth=0.4pt](-2,-0.12)(2,-0.12)

\rput(2.2, 1.8){\tiny $-8$}
\rput(2.2, 1.56){\tiny $-7$}
\rput(2.2, 1.32){\tiny $-6$}
\rput(2.2, 1.08){\tiny $-5$}
\rput(2.2, 0.84){\tiny $-4$}
\rput(2.2, 0.6){\tiny $-3$}
\rput(2.2, 0.36){\tiny $-2$}
\rput(2.2, 0.12){\tiny $-1$}
\rput(2.2, -0.12){\tiny $1$}
\rput(2.2, -0.36){\tiny $2$}
\rput(2.2, -0.6){\tiny $3$}
\rput(2.2, -0.84){\tiny $4$}
\rput(2.2, -1.08){\tiny $5$}
\rput(2.2, -1.32){\tiny $6$}
\rput(2.2, -1.56){\tiny $7$}
\rput(2.2, -1.8){\tiny $8$}

\psline[linestyle=dotted, linewidth=0.4pt](2.5,1.8)(4.5,1.8)
\psline[linestyle=dotted, linewidth=0.4pt](2.5,1.56)(4.5,1.56)
\psline[linestyle=dotted, linewidth=0.4pt](2.5,1.32)(4.5,1.32)
\psline[linestyle=dotted, linewidth=0.4pt](2.5,1.08)(4.5,1.08)
\psline[linestyle=dotted, linewidth=0.4pt](2.5,0.84)(4.5,0.84)
\psline[linestyle=dotted, linewidth=0.4pt](2.5,0.6)(4.5,0.6)
\psline[linestyle=dotted, linewidth=0.4pt](2.5,0.36)(4.5,0.36)
\psline[linestyle=dotted, linewidth=0.4pt](2.5,0.12)(4.5,0.12)
\psline[linestyle=dotted, linewidth=0.4pt](2.5,-1.8)(4.5,-1.8)
\psline[linestyle=dotted, linewidth=0.4pt](2.5,-1.56)(4.5,-1.56)
\psline[linestyle=dotted, linewidth=0.4pt](2.5,-1.32)(4.5,-1.32)
\psline[linestyle=dotted, linewidth=0.4pt](2.5,-1.08)(4.5,-1.08)
\psline[linestyle=dotted, linewidth=0.4pt](2.5,-0.84)(4.5,-0.84)
\psline[linestyle=dotted, linewidth=0.4pt](2.5,-0.6)(4.5,-0.6)
\psline[linestyle=dotted, linewidth=0.4pt](2.5,-0.36)(4.5,-0.36)
\psline[linestyle=dotted, linewidth=0.4pt](2.5,-0.12)(4.5,-0.12)

\rput(2.2, 1.8){\tiny $-8$}
\rput(2.2, 1.56){\tiny $-7$}
\rput(2.2, 1.32){\tiny $-6$}
\rput(2.2, 1.08){\tiny $-5$}
\rput(2.2, 0.84){\tiny $-4$}
\rput(2.2, 0.6){\tiny $-3$}
\rput(2.2, 0.36){\tiny $-2$}
\rput(2.2, 0.12){\tiny $-1$}
\rput(2.2, -0.12){\tiny $1$}
\rput(2.2, -0.36){\tiny $2$}
\rput(2.2, -0.6){\tiny $3$}
\rput(2.2, -0.84){\tiny $4$}
\rput(2.2, -1.08){\tiny $5$}
\rput(2.2, -1.32){\tiny $6$}
\rput(2.2, -1.56){\tiny $7$}
\rput(2.2, -1.8){\tiny $8$}

\psline(5,-0.9)(5,0.9)
\psline(5.3,-0.9)(5.3,0.9)
\psline(5.6,-0.9)(5.6,0.9)
\psline(6.2,-0.9)(6.2,0.9)
\psline(7.4,-0.9)(7.4,0.9)
\psline(8.9,-0.9)(8.9,0.9)
\psline(9.2,-0.9)(9.2,0.9)
\psline(9.5,-0.9)(9.5,0.9)
\psline[linecolor=white, linewidth=2pt](5.9, 0.9)(8,-0.9)
\psline[linecolor=red](5.9, 0.9)(8,-0.9)
\psline[linecolor=white, linewidth=2pt](5.9, -0.9)(8,0.9)
\psline[linecolor=red](5.9, -0.9)(8,0.9)

\pscurve[linecolor=white, linewidth=2pt](6.8,0.9)(6.65,0)(6.8,-0.9)
\pscurve(6.8,0.9)(6.65,0)(6.8,-0.9)

\pscurve[linecolor=white, linewidth=2pt](7.7,0.9)(7.85,0)(7.7,-0.9)
\pscurve(7.7,0.9)(7.85,0)(7.7,-0.9)

\psline[linecolor=white, linewidth=2pt](6.5, 0.9)(8.6,-0.9)
\psline[linecolor=red](6.5, 0.9)(8.6,-0.9)
\psline[linecolor=white, linewidth=2pt](6.5, -0.9)(8.6,0.9)
\psline[linecolor=red](6.5, -0.9)(8.6,0.9)

\psline[linecolor=white, linewidth=2pt](7.1,0.9)(7.1,-0.9)
\psline(7.1,0.9)(7.1,-0.9)

\psline[linecolor=white, linewidth=2pt](8.3,0.9)(8.3,-0.9)
\psline(8.3,0.9)(8.3,-0.9)

\psdots[dotsize=2.5pt](5,0.9)(5.3,0.9)(5.6,0.9)(5.9,0.9)(6.2,0.9)(6.5,0.9)(6.8,0.9)(7.1,0.9)(7.4,0.9)(7.7,0.9)(8,0.9)(8.3,0.9)(8.6,0.9)(8.9,0.9)(9.2,0.9)(9.5,0.9)

\psdots[dotsize=2.5pt](5,-0.9)(5.3,-0.9)(5.6,-0.9)(5.9,-0.9)(6.2,-0.9)(6.5,-0.9)(6.8,-0.9)(7.1,-0.9)(7.4,-0.9)(7.7,-0.9)(8,-0.9)(8.3,-0.9)(8.6,-0.9)(8.9,-0.9)(9.2,-0.9)(9.5,-0.9)

\end{pspicture}}

\caption{The diagram $N_t$ for $t=(3,-5)(-3,5)$ and the braid $\beta_t$.}
\label{figure:ref2}

\end{figure}
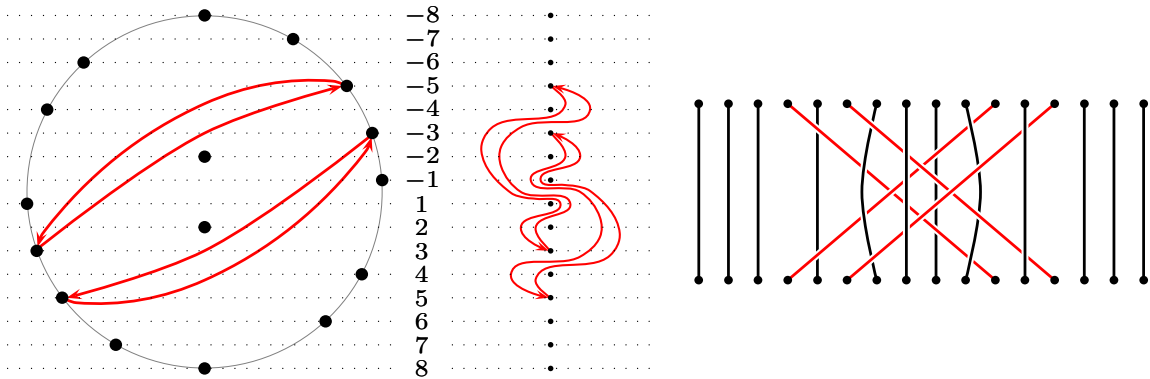


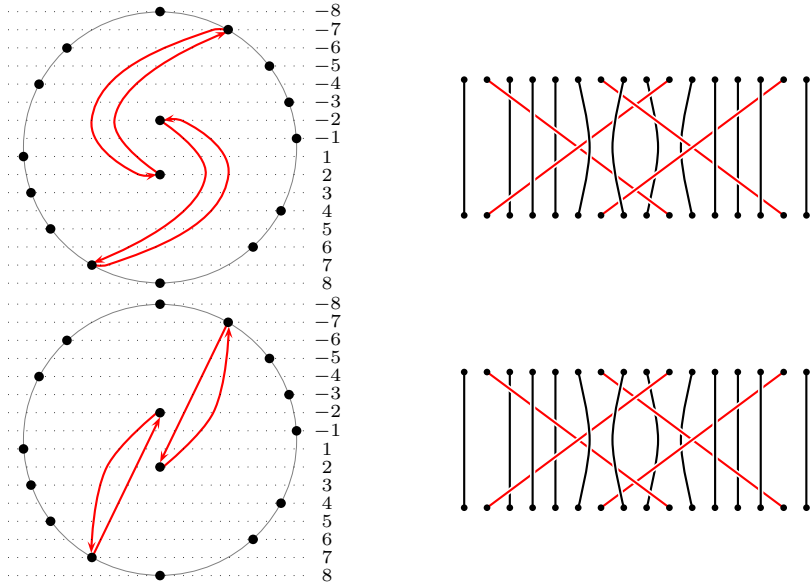
\begin{figure}[h!]

\begin{pspicture}(-2,-1.92)(10,1.92)
\pscircle[linecolor=gray, linewidth=0.2pt](0,0){1.8}

\pscurve[linecolor=red]{->}(0,0.36)(0.6,-0.36)(-0.88,-1.53)
\pscurve[linecolor=red]{<-}(0.04,0.36)(0.3, 0.35)(0.9,-0.36)(-0.7, -1.56)(-0.897,-1.56)

\pscurve[linecolor=red]{->}(0,-0.36)(-0.6,0.36)(0.88,1.53)
\pscurve[linecolor=red]{<-}(-0.04,-0.36)(-0.3,-0.35)(-0.9,0.36)(0.7, 1.56)(0.897,1.56)

\psdots(0,1.8)(0.897,1.56)(-1.2237,1.32)(1.44,1.08)(-1.59197,0.84)(1.697,0.6)(0,0.36)(1.796,0.12)(-1.796,-0.12)(0,-0.36)(-1.697,-0.6)(1.59197,-0.84)(-1.44,-1.08)(1.2237,-1.32)(-0.897,-1.56)(0,-1.8)


\psline[linestyle=dotted, linewidth=0.4pt](-2,1.8)(2,1.8)
\psline[linestyle=dotted, linewidth=0.4pt](-2,1.56)(2,1.56)
\psline[linestyle=dotted, linewidth=0.4pt](-2,1.32)(2,1.32)
\psline[linestyle=dotted, linewidth=0.4pt](-2,1.08)(2,1.08)
\psline[linestyle=dotted, linewidth=0.4pt](-2,0.84)(2,0.84)
\psline[linestyle=dotted, linewidth=0.4pt](-2,0.6)(2,0.6)
\psline[linestyle=dotted, linewidth=0.4pt](-2,0.36)(2,0.36)
\psline[linestyle=dotted, linewidth=0.4pt](-2,0.12)(2,0.12)
\psline[linestyle=dotted, linewidth=0.4pt](-2,-1.8)(2,-1.8)
\psline[linestyle=dotted, linewidth=0.4pt](-2,-1.56)(2,-1.56)
\psline[linestyle=dotted, linewidth=0.4pt](-2,-1.32)(2,-1.32)
\psline[linestyle=dotted, linewidth=0.4pt](-2,-1.08)(2,-1.08)
\psline[linestyle=dotted, linewidth=0.4pt](-2,-0.84)(2,-0.84)
\psline[linestyle=dotted, linewidth=0.4pt](-2,-0.6)(2,-0.6)
\psline[linestyle=dotted, linewidth=0.4pt](-2,-0.36)(2,-0.36)
\psline[linestyle=dotted, linewidth=0.4pt](-2,-0.12)(2,-0.12)

\rput(2.2, 1.8){\tiny $-8$}
\rput(2.2, 1.56){\tiny $-7$}
\rput(2.2, 1.32){\tiny $-6$}
\rput(2.2, 1.08){\tiny $-5$}
\rput(2.2, 0.84){\tiny $-4$}
\rput(2.2, 0.6){\tiny $-3$}
\rput(2.2, 0.36){\tiny $-2$}
\rput(2.2, 0.12){\tiny $-1$}
\rput(2.2, -0.12){\tiny $1$}
\rput(2.2, -0.36){\tiny $2$}
\rput(2.2, -0.6){\tiny $3$}
\rput(2.2, -0.84){\tiny $4$}
\rput(2.2, -1.08){\tiny $5$}
\rput(2.2, -1.32){\tiny $6$}
\rput(2.2, -1.56){\tiny $7$}
\rput(2.2, -1.8){\tiny $8$}

\psline(4,0.9)(4, -0.9)
\psline(8.5,0.9)(8.5,-0.9)
\psline(4.6,0.9)(4.6,-0.9)
\psline(5.2,0.9)(5.2,-0.9)
\pscurve(6.4,0.9)(6.55,0)(6.4,-0.9)
\pscurve(7,0.9)(6.85,0)(7,-0.9)
\psline(7.6,0.9)(7.6,-0.9)

\psline[linecolor=white, linewidth=2pt](4.3, 0.9)(6.7,-0.9)
\psline[linecolor=red](4.3, 0.9)(6.7,-0.9)
\psline[linecolor=white, linewidth=2pt](4.3, -0.9)(6.7,0.9)
\psline[linecolor=red](4.3, -0.9)(6.7,0.9)

\psline[linecolor=white, linewidth=2pt](5.8, 0.9)(8.2,-0.9)
\psline[linecolor=red](5.8, 0.9)(8.2,-0.9)
\psline[linecolor=white, linewidth=2pt](5.8, -0.9)(8.2,0.9)
\psline[linecolor=red](5.8, -0.9)(8.2,0.9)

\psline[linecolor=white, linewidth=2pt](4.9,0.9)(4.9,-0.9)
\psline(4.9,0.9)(4.9,-0.9)

\pscurve[linecolor=white, linewidth=2pt](5.5,0.9)(5.65,0)(5.5,-0.9)
\pscurve(5.5,0.9)(5.65,0)(5.5,-0.9)

\pscurve[linecolor=white, linewidth=2pt](6.1,0.9)(5.95,0)(6.1,-0.9)
\pscurve(6.1,0.9)(5.95,0)(6.1,-0.9)

\psline[linecolor=white, linewidth=2pt](7.3,0.9)(7.3,-0.9)
\psline(7.3,0.9)(7.3,-0.9)

\psline[linecolor=white, linewidth=2pt](7.9,0.9)(7.9,-0.9)
\psline(7.9,0.9)(7.9,-0.9)

\psdots[dotsize=2.5pt](4,0.9)(4.3,0.9)(4.6,0.9)(4.9,0.9)(5.2,0.9)(5.5,0.9)(5.8,0.9)(6.1,0.9)(6.4,0.9)(6.7,0.9)(7,0.9)(7.3,0.9)(7.6,0.9)(7.9,0.9)(8.2,0.9)(8.5,0.9)

\psdots[dotsize=2.5pt](4,-0.9)(4.3,-0.9)(4.6,-0.9)(4.9,-0.9)(5.2,-0.9)(5.5,-0.9)(5.8,-0.9)(6.1,-0.9)(6.4,-0.9)(6.7,-0.9)(7,-0.9)(7.3,-0.9)(7.6,-0.9)(7.9,-0.9)(8.2,-0.9)(8.5,-0.9)

\end{pspicture}

\begin{pspicture}(-2,-1.92)(10,1.92)
\pscircle[linecolor=gray, linewidth=0.2pt](0,0){1.8}

\psline[linecolor=red]{<-}(0,0.3)(-0.88,-1.53)
\pscurve[linecolor=red]{->}(-0.04,0.36)(-0.7,-0.36)(-0.905,-1.5)

\psline[linecolor=red]{<-}(0,-0.3)(0.88,1.53)
\pscurve[linecolor=red]{->}(0.04,-0.36)(0.7,0.36)(0.905,1.5)

\psdots(0,1.8)(0.897,1.56)(-1.2237,1.32)(1.44,1.08)(-1.59197,0.84)(1.697,0.6)(0,0.36)(1.796,0.12)(-1.796,-0.12)(0,-0.36)(-1.697,-0.6)(1.59197,-0.84)(-1.44,-1.08)(1.2237,-1.32)(-0.897,-1.56)(0,-1.8)


\psline[linestyle=dotted, linewidth=0.4pt](-2,1.8)(2,1.8)
\psline[linestyle=dotted, linewidth=0.4pt](-2,1.56)(2,1.56)
\psline[linestyle=dotted, linewidth=0.4pt](-2,1.32)(2,1.32)
\psline[linestyle=dotted, linewidth=0.4pt](-2,1.08)(2,1.08)
\psline[linestyle=dotted, linewidth=0.4pt](-2,0.84)(2,0.84)
\psline[linestyle=dotted, linewidth=0.4pt](-2,0.6)(2,0.6)
\psline[linestyle=dotted, linewidth=0.4pt](-2,0.36)(2,0.36)
\psline[linestyle=dotted, linewidth=0.4pt](-2,0.12)(2,0.12)
\psline[linestyle=dotted, linewidth=0.4pt](-2,-1.8)(2,-1.8)
\psline[linestyle=dotted, linewidth=0.4pt](-2,-1.56)(2,-1.56)
\psline[linestyle=dotted, linewidth=0.4pt](-2,-1.32)(2,-1.32)
\psline[linestyle=dotted, linewidth=0.4pt](-2,-1.08)(2,-1.08)
\psline[linestyle=dotted, linewidth=0.4pt](-2,-0.84)(2,-0.84)
\psline[linestyle=dotted, linewidth=0.4pt](-2,-0.6)(2,-0.6)
\psline[linestyle=dotted, linewidth=0.4pt](-2,-0.36)(2,-0.36)
\psline[linestyle=dotted, linewidth=0.4pt](-2,-0.12)(2,-0.12)

\rput(2.2, 1.8){\tiny $-8$}
\rput(2.2, 1.56){\tiny $-7$}
\rput(2.2, 1.32){\tiny $-6$}
\rput(2.2, 1.08){\tiny $-5$}
\rput(2.2, 0.84){\tiny $-4$}
\rput(2.2, 0.6){\tiny $-3$}
\rput(2.2, 0.36){\tiny $-2$}
\rput(2.2, 0.12){\tiny $-1$}
\rput(2.2, -0.12){\tiny $1$}
\rput(2.2, -0.36){\tiny $2$}
\rput(2.2, -0.6){\tiny $3$}
\rput(2.2, -0.84){\tiny $4$}
\rput(2.2, -1.08){\tiny $5$}
\rput(2.2, -1.32){\tiny $6$}
\rput(2.2, -1.56){\tiny $7$}
\rput(2.2, -1.8){\tiny $8$}

\psline(4,0.9)(4, -0.9)
\psline(8.5,0.9)(8.5,-0.9)
\psline(4.6,0.9)(4.6,-0.9)
\psline(5.2,0.9)(5.2,-0.9)
\pscurve(6.4,0.9)(6.55,0)(6.4,-0.9)
\pscurve(7,0.9)(6.85,0)(7,-0.9)
\psline(7.6,0.9)(7.6,-0.9)

\psline[linecolor=white, linewidth=2pt](5.8, 0.9)(8.2,-0.9)
\psline[linecolor=red](5.8, 0.9)(8.2,-0.9)
\psline[linecolor=white, linewidth=2pt](5.8, -0.9)(8.2,0.9)
\psline[linecolor=red](5.8, -0.9)(8.2,0.9)

\psline[linecolor=white, linewidth=2pt](4.3, 0.9)(6.7,-0.9)
\psline[linecolor=red](4.3, 0.9)(6.7,-0.9)
\psline[linecolor=white, linewidth=2pt](4.3, -0.9)(6.7,0.9)
\psline[linecolor=red](4.3, -0.9)(6.7,0.9)

\psline[linecolor=white, linewidth=2pt](4.9,0.9)(4.9,-0.9)
\psline(4.9,0.9)(4.9,-0.9)

\pscurve[linecolor=white, linewidth=2pt](5.5,0.9)(5.65,0)(5.5,-0.9)
\pscurve(5.5,0.9)(5.65,0)(5.5,-0.9)

\pscurve[linecolor=white, linewidth=2pt](6.1,0.9)(5.95,0)(6.1,-0.9)
\pscurve(6.1,0.9)(5.95,0)(6.1,-0.9)

\psline[linecolor=white, linewidth=2pt](7.3,0.9)(7.3,-0.9)
\psline(7.3,0.9)(7.3,-0.9)

\psline[linecolor=white, linewidth=2pt](7.9,0.9)(7.9,-0.9)
\psline(7.9,0.9)(7.9,-0.9)

\psdots[dotsize=2.5pt](4,0.9)(4.3,0.9)(4.6,0.9)(4.9,0.9)(5.2,0.9)(5.5,0.9)(5.8,0.9)(6.1,0.9)(6.4,0.9)(6.7,0.9)(7,0.9)(7.3,0.9)(7.6,0.9)(7.9,0.9)(8.2,0.9)(8.5,0.9)

\psdots[dotsize=2.5pt](4,-0.9)(4.3,-0.9)(4.6,-0.9)(4.9,-0.9)(5.2,-0.9)(5.5,-0.9)(5.8,-0.9)(6.1,-0.9)(6.4,-0.9)(6.7,-0.9)(7,-0.9)(7.3,-0.9)(7.6,-0.9)(7.9,-0.9)(8.2,-0.9)(8.5,-0.9)

\end{pspicture}

\caption{Two diagrams $N_t$ for $t=(2,-7)(-2,7)$ and the corresponding Artin braids $\beta_t$. Note that the two Artin braids on the right are equal in the quotient $\widetilde{A}_{B_n}$.}
\label{figure:ref3}

\end{figure}

We now generalize the above picturial process, by associating a (possibly non unique) noncrossing diagram $N_x$ and an Artin braid $\beta_x$ to \textit{every} $x\in\NC(W_{D_n}, c)$. 

\subsubsection{Pictures for noncrossing partitions}\label{pic:all}

To obtain a noncrossing diagram $N_x$ with oriented curvilinear polygons from $x$ as we did for reflections in the previous section, we proceed as follows: we orient every polygon of the noncrossing partition in counterclockwise order (note that this is the opposite orientation to the one given by the corresponding cycle of $x$, that is, an arrow $j_2\rightarrow j_1$ means that the cycle of $x$ sends $j_1$ to $j_2$; hence this orientation corresponds to $x^{-1}$). Polygons reduced to a single edge are replaced by curvilinear polygons with two edges as we did for reflections in Section~\ref{pic:ref}. Again we split the points with labels $\pm i_1$ into two points with labels $-i_1$ and $i_1$ respectively as in Figure~\ref{figure:ordre}. 

In the case where the middle point in the Athanasiadis-Reiner model has no edge starting at it and does not lie inside a symmetric polygon, then the two points $i_1$ and $-i_1$ have no edge starting at them in the new diagram. In the case where there are two distinct polygons $P$ and $-P$ sharing the middle point, they are separated so
 that each point lies in the correct curvilinear polygon (see Figure~\ref{figure:split}): there might be several non-isotopic diagrams which work when separating $P$
  from $-P$ (in case $P$ is a $2$-cycle we precisely get what we already noticed and explained in Figure~\ref{figure:ref3}).  
A similar argument to the one given in Figure~\ref{figure:ref3} shows that the images in $\Ab$ of the various Artin braids $\beta_x$ obtained from the distinct diagrams $N_x$ at the end of the process explained below will be equal. In case there is a symmetric polygon $P=-P$ or a factor $[j][i_1]$ in $x$, we add a curvilinear polygon with two edges joining $-i_1$ to $i_1$, oriented in counterclockwise order (Recall that in the noncrossing representation of $[j][i_1]$, the factor $[j]$ is already represented by a curvilinear ``polygon'' with two edges and the point $i_1$ inside it. Here we orient this polygon in counterclockwise order as in all other cases).   

\begin{figure}[h!]

\begin{tabular}{cc}

\begin{pspicture}(-2,-1.92)(2.5,1.8)
\pscircle[linecolor=gray, linewidth=0.2pt](0,0){1.8}

\psline[linecolor=red](0.897,1.56)(1.44,1.08)(0,0)(0.897,1.56)
\psline[linecolor=red](-0.897,-1.56)(-1.44,-1.08)(0,0)(-0.897,-1.56)
\psline[linecolor=red](0,1.8)(-1.796,-0.12)
\psline[linecolor=red](0,-1.8)(1.796,0.12)

\psdots(0,1.8)(0.897,1.56)(-1.2237,1.32)(1.44,1.08)(-1.59197,0.84)(1.697,0.6)(0,0)(1.796,0.12)(-1.796,-0.12)(-1.697,-0.6)(1.59197,-0.84)(-1.44,-1.08)(1.2237,-1.32)(-0.897,-1.56)(0,-1.8)


\psline[linestyle=dotted, linewidth=0.4pt](-2,1.8)(2,1.8)
\psline[linestyle=dotted, linewidth=0.4pt](-2,1.56)(2,1.56)
\psline[linestyle=dotted, linewidth=0.4pt](-2,1.32)(2,1.32)
\psline[linestyle=dotted, linewidth=0.4pt](-2,1.08)(2,1.08)
\psline[linestyle=dotted, linewidth=0.4pt](-2,0.84)(2,0.84)
\psline[linestyle=dotted, linewidth=0.4pt](-2,0.6)(2,0.6)
\psline[linestyle=dotted, linewidth=0.4pt](-2,0.36)(2,0.36)
\psline[linestyle=dotted, linewidth=0.4pt](-2,0.12)(2,0.12)
\psline[linestyle=dotted, linewidth=0.4pt](-2,-1.8)(2,-1.8)
\psline[linestyle=dotted, linewidth=0.4pt](-2,-1.56)(2,-1.56)
\psline[linestyle=dotted, linewidth=0.4pt](-2,-1.32)(2,-1.32)
\psline[linestyle=dotted, linewidth=0.4pt](-2,-1.08)(2,-1.08)
\psline[linestyle=dotted, linewidth=0.4pt](-2,-0.84)(2,-0.84)
\psline[linestyle=dotted, linewidth=0.4pt](-2,-0.6)(2,-0.6)
\psline[linestyle=dotted, linewidth=0.4pt](-2,-0.36)(2,-0.36)
\psline[linestyle=dotted, linewidth=0.4pt](-2,-0.12)(2,-0.12)

\rput(2.2, 1.8){\tiny $-8$}
\rput(2.2, 1.56){\tiny $-7$}
\rput(2.2, 1.32){\tiny $-6$}
\rput(2.2, 1.08){\tiny $-5$}
\rput(2.2, 0.84){\tiny $-4$}
\rput(2.2, 0.6){\tiny $-3$}
\rput(-0.4,0){\tiny $-2$}
\rput(0.4,0){\tiny $2$}
\rput(2.2, 0.12){\tiny $-1$}
\rput(2.2, -0.12){\tiny $1$}
\rput(2.2, -0.6){\tiny $3$}
\rput(2.2, -0.84){\tiny $4$}
\rput(2.2, -1.08){\tiny $5$}
\rput(2.2, -1.32){\tiny $6$}
\rput(2.2, -1.56){\tiny $7$}
\rput(2.2, -1.8){\tiny $8$}
\end{pspicture}

&

\begin{pspicture}(-2,-1.92)(2.5,1.8)
\pscircle[linecolor=gray, linewidth=0.2pt](0,0){1.8}

\psline[linecolor=red]{->}(-0.897,-1.56)(0,0.32)
\pscurve[linecolor=red]{<-}(-1.44,-1.04)(-0.8,-0.35)(0,0.36)
\pscurve[linecolor=red]{<-}(-0.92,-1.53)(-1.15,-1.1)(-1.44,-1.08)

\psline[linecolor=red]{->}(0.897,1.56)(0,-0.32)
\pscurve[linecolor=red]{<-}(1.44,1.04)(0.8,0.35)(0,-0.36)
\pscurve[linecolor=red]{<-}(0.92,1.53)(1.15,1.1)(1.44,1.08)

\pscurve[linecolor=red]{->}(0,1.8)(-1.1,0.85)(-1.796,-0.08)
\pscurve[linecolor=red]{<-}(0.02,1.76)(-0.65,0.85)(-1.78,-0.12)

\pscurve[linecolor=red]{->}(0,-1.8)(1.1,-0.85)(1.796,0.08)
\pscurve[linecolor=red]{<-}(-0.02,-1.76)(0.65,-0.85)(1.78,0.12)

\psdots(0,1.8)(0.897,1.56)(-1.2237,1.32)(1.44,1.08)(-1.59197,0.84)(1.697,0.6)(0,0.36)(1.796,0.12)(-1.796,-0.12)(0,-0.36)(-1.697,-0.6)(1.59197,-0.84)(-1.44,-1.08)(1.2237,-1.32)(-0.897,-1.56)(0,-1.8)


\psline[linestyle=dotted, linewidth=0.4pt](-2,1.8)(2,1.8)
\psline[linestyle=dotted, linewidth=0.4pt](-2,1.56)(2,1.56)
\psline[linestyle=dotted, linewidth=0.4pt](-2,1.32)(2,1.32)
\psline[linestyle=dotted, linewidth=0.4pt](-2,1.08)(2,1.08)
\psline[linestyle=dotted, linewidth=0.4pt](-2,0.84)(2,0.84)
\psline[linestyle=dotted, linewidth=0.4pt](-2,0.6)(2,0.6)
\psline[linestyle=dotted, linewidth=0.4pt](-2,0.36)(2,0.36)
\psline[linestyle=dotted, linewidth=0.4pt](-2,0.12)(2,0.12)
\psline[linestyle=dotted, linewidth=0.4pt](-2,-1.8)(2,-1.8)
\psline[linestyle=dotted, linewidth=0.4pt](-2,-1.56)(2,-1.56)
\psline[linestyle=dotted, linewidth=0.4pt](-2,-1.32)(2,-1.32)
\psline[linestyle=dotted, linewidth=0.4pt](-2,-1.08)(2,-1.08)
\psline[linestyle=dotted, linewidth=0.4pt](-2,-0.84)(2,-0.84)
\psline[linestyle=dotted, linewidth=0.4pt](-2,-0.6)(2,-0.6)
\psline[linestyle=dotted, linewidth=0.4pt](-2,-0.36)(2,-0.36)
\psline[linestyle=dotted, linewidth=0.4pt](-2,-0.12)(2,-0.12)

\rput(2.2, 1.8){\tiny $-8$}
\rput(2.2, 1.56){\tiny $-7$}
\rput(2.2, 1.32){\tiny $-6$}
\rput(2.2, 1.08){\tiny $-5$}
\rput(2.2, 0.84){\tiny $-4$}
\rput(2.2, 0.6){\tiny $-3$}
\rput(2.2, 0.36){\tiny $-2$}
\rput(2.2, 0.12){\tiny $-1$}
\rput(2.2, -0.12){\tiny $1$}
\rput(2.2, -0.36){\tiny $2$}
\rput(2.2, -0.6){\tiny $3$}
\rput(2.2, -0.84){\tiny $4$}
\rput(2.2, -1.08){\tiny $5$}
\rput(2.2, -1.32){\tiny $6$}
\rput(2.2, -1.56){\tiny $7$}
\rput(2.2, -1.8){\tiny $8$}
\end{pspicture}

\end{tabular}

\caption{Splitting of two polygons with common middle point.}
\label{figure:split}

\end{figure}
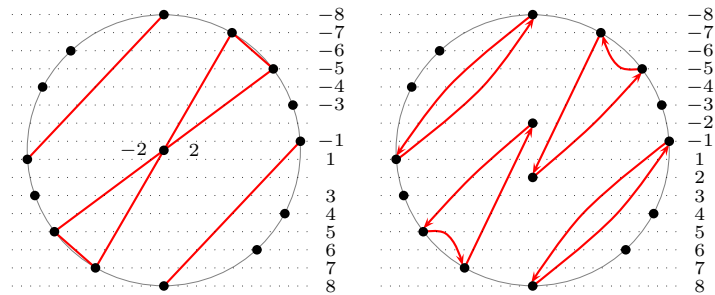 

If one has the diagram $N_x$ with oriented curvilinear polygons as in Figure~\ref{figure:split} on the right, we proceed exactly as we did for reflections in Section~\ref{pic:ref} to obtain $\beta_x$: firstly, we put all the black points on a vertical line and project the noncrossing diagram to obtain a picture as in the left pictures in Figures~\ref{figure:ref1} and \ref{figure:ref2}; this diagram gives the Artin braid $\beta_x$ viewed from the bottom. We illustrate this process for the noncrossing diagram of the element $x_2=((8,7,5))[6,3,1][2]$ of Figure~\ref{figure:ref4} in Figure~\ref{figure:split_braid}. Note that as a consequence of this procedure, the orientation we put on polygons, which as we already noticed at the beginning of the subsection is not the one corresponding to $x$ but to $x^{-1}$, defines the permutation induced by the strands of $\beta_x$. The fact that the permutation induced by the strands of $\beta_x$ is $x^{-1}$ rather than $x$ comes from the fact that our convention is to concatenate Artin braids from top to bottom. 

Note that we can always recover the braid from the middle diagram without ambiguity, because all the strands either strictly go up or down, except possibly in one case: in case $i_1= 2$ and $x=(1,-1)(2,-2)\in \NC(W_{D_n},c)$, then the strands joining $1$ to $-1$ and $-1$ to $1$ do not strictly go up or down. In that case we represent the braid as done in Figure~\ref{figure:1_2} in the next subsection.

\begin{figure}[h!]

\centering
\psscalebox{1.8}{\begin{pspicture}(2,-1.92)(10,1.92)

\pscurve[linecolor=red, linewidth=0.5pt]{->}(3.5,-0.36)(3.8,-0.12)(3.4,0.12)(3.5,0.34)
\pscurve[linecolor=red, linewidth=0.5pt]{<-}(3.5,-0.34)(3.6,-0.12)(3.2,0.12)(3.5,0.36)

\pscurve[linecolor=blue, linewidth=0.5pt]{->}(3.5,0.12)(3.8,0.36)(3.52,0.58)
\pscurve[linecolor=blue, linewidth=0.5pt]{->}(3.5,0.6)(3.8,0.9)(3.35,1.13)(3.5,1.3)
\pscurve[linecolor=blue, linewidth=0.5pt]{->}(3.49, 1.32)(2.8,1.08)(3.65,0.84)(2.8,0.36)(3.47,-0.12)

\pscurve[linecolor=blue, linewidth=0.5pt]{->}(3.5,-0.12)(3.2,-0.36)(3.48,-0.58)
\pscurve[linecolor=blue, linewidth=0.5pt]{->}(3.5,-0.6)(3.2,-0.9)(3.65,-1.13)(3.5,-1.3)
\pscurve[linecolor=blue, linewidth=0.5pt]{->}(3.51,-1.32)(4.2,-1.08)(3.35,-0.84)(4.2,-0.36)(3.53,0.12)

\pscurve[linecolor=red, linewidth=0.5pt]{<-}(3.5,1.11)(3.65,1.32)(3.35,1.56)(3.5,1.8)
\pscurve[linecolor=red, linewidth=0.5pt]{->}(3.5,1.08)(3.9,1.32)(3.52,1.55)
\psline[linecolor=red, linewidth=0.5pt]{->}(3.5,1.56)(3.5,1.77)

\pscurve[linecolor=red, linewidth=0.5pt]{<-}(3.5,-1.11)(3.35,-1.32)(3.65,-1.56)(3.5,-1.8)
\pscurve[linecolor=red, linewidth=0.5pt]{->}(3.5,-1.08)(3.1,-1.32)(3.48,-1.55)
\psline[linecolor=red, linewidth=0.5pt]{->}(3.5,-1.56)(3.5,-1.77)


\psdots[dotsize=1.6pt](3.5,1.8)(3.5,1.56)(3.5,1.32)(3.5,1.08)(3.5,0.84)(3.5,0.6)(3.5,0.36)(3.5,0.12)(3.5,-0.12)(3.5,-0.36)(3.5,-0.6)(3.5,-0.84)(3.5,-1.08)(3.5,-1.32)(3.5,-1.56)(3.5,-1.8)

\rput(2.2, 1.8){\tiny $-8$}
\rput(2.2, 1.56){\tiny $-7$}
\rput(2.2, 1.32){\tiny $-6$}
\rput(2.2, 1.08){\tiny $-5$}
\rput(2.2, 0.84){\tiny $-4$}
\rput(2.2, 0.6){\tiny $-3$}
\rput(2.2, 0.36){\tiny $-2$}
\rput(2.2, 0.12){\tiny $-1$}
\rput(2.2, -0.12){\tiny $1$}
\rput(2.2, -0.36){\tiny $2$}
\rput(2.2, -0.6){\tiny $3$}
\rput(2.2, -0.84){\tiny $4$}
\rput(2.2, -1.08){\tiny $5$}
\rput(2.2, -1.32){\tiny $6$}
\rput(2.2, -1.56){\tiny $7$}
\rput(2.2, -1.8){\tiny $8$}

\psline[linestyle=dotted, linewidth=0.4pt](2.5,1.8)(4.5,1.8)
\psline[linestyle=dotted, linewidth=0.4pt](2.5,1.56)(4.5,1.56)
\psline[linestyle=dotted, linewidth=0.4pt](2.5,1.32)(4.5,1.32)
\psline[linestyle=dotted, linewidth=0.4pt](2.5,1.08)(4.5,1.08)
\psline[linestyle=dotted, linewidth=0.4pt](2.5,0.84)(4.5,0.84)
\psline[linestyle=dotted, linewidth=0.4pt](2.5,0.6)(4.5,0.6)
\psline[linestyle=dotted, linewidth=0.4pt](2.5,0.36)(4.5,0.36)
\psline[linestyle=dotted, linewidth=0.4pt](2.5,0.12)(4.5,0.12)
\psline[linestyle=dotted, linewidth=0.4pt](2.5,-1.8)(4.5,-1.8)
\psline[linestyle=dotted, linewidth=0.4pt](2.5,-1.56)(4.5,-1.56)
\psline[linestyle=dotted, linewidth=0.4pt](2.5,-1.32)(4.5,-1.32)
\psline[linestyle=dotted, linewidth=0.4pt](2.5,-1.08)(4.5,-1.08)
\psline[linestyle=dotted, linewidth=0.4pt](2.5,-0.84)(4.5,-0.84)
\psline[linestyle=dotted, linewidth=0.4pt](2.5,-0.6)(4.5,-0.6)
\psline[linestyle=dotted, linewidth=0.4pt](2.5,-0.36)(4.5,-0.36)
\psline[linestyle=dotted, linewidth=0.4pt](2.5,-0.12)(4.5,-0.12)

\rput(2.2, 1.8){\tiny $-8$}
\rput(2.2, 1.56){\tiny $-7$}
\rput(2.2, 1.32){\tiny $-6$}
\rput(2.2, 1.08){\tiny $-5$}
\rput(2.2, 0.84){\tiny $-4$}
\rput(2.2, 0.6){\tiny $-3$}
\rput(2.2, 0.36){\tiny $-2$}
\rput(2.2, 0.12){\tiny $-1$}
\rput(2.2, -0.12){\tiny $1$}
\rput(2.2, -0.36){\tiny $2$}
\rput(2.2, -0.6){\tiny $3$}
\rput(2.2, -0.84){\tiny $4$}
\rput(2.2, -1.08){\tiny $5$}
\rput(2.2, -1.32){\tiny $6$}
\rput(2.2, -1.56){\tiny $7$}
\rput(2.2, -1.8){\tiny $8$}

\pscurve(6.2,-0.9)(6.5,0.5)(6.2,0.9)

\pscurve[linecolor=red](8.6,0.9)(8.6,0.6)(9.2,-0.9)

\psline[linecolor=red](9.2,0.9)(9.5,-0.9)
\psline[linecolor=white, linewidth=2pt](9.5,0.9)(8.6,-0.9)
\psline[linecolor=red](9.5,0.9)(8.6,-0.9)

\psline[linecolor=white, linewidth=2pt](5.6,0.9)(7.4,-0.9)
\psline[linecolor=blue](5.6,0.9)(7.4,-0.9)

\psline[linecolor=white, linewidth=2pt](6.5,0.9)(5.6,-0.9)
\psline[linecolor=blue](6.5,0.9)(5.6,-0.9)

\psline[linecolor=white, linewidth=2pt](6.8,0.9)(7.7,-0.9)
\psline[linecolor=red](6.8,0.9)(7.7,-0.9)

\psline[linecolor=white, linewidth=2pt](7.4,0.9)(8,-0.9)
\psline[linecolor=blue](7.4,0.9)(8,-0.9)

\psline[linecolor=white, linewidth=2pt](8,0.9)(8.9,-0.9)
\psline[linecolor=blue](8,0.9)(8.9,-0.9)

\psline[linecolor=white, linewidth=2pt](8.9,0.9)(7.1,-0.9)
\psline[linecolor=blue](8.9,0.9)(7.1,-0.9)

\psline[linecolor=white, linewidth=2pt](7.7,0.9)(6.8,-0.9)
\psline[linecolor=red](7.7,0.9)(6.8,-0.9)

\psline[linecolor=white, linewidth=2pt](7.1,0.9)(6.5,-0.9)
\psline[linecolor=blue](7.1,0.9)(6.5,-0.9)

\pscurve[linecolor=white, linewidth=2pt](8.3,0.9)(8,0.5)(8.3,-0.9)
\pscurve(8.3,0.9)(8,0.5)(8.3,-0.9)

\psline[linecolor=white, linewidth=2pt](5,0.9)(5.9,-0.9)
\psline[linecolor=red](5,0.9)(5.9,-0.9)

\pscurve[linecolor=white, linewidth=2pt](5.9,0.9)(5.9,0.6)(5.3,-0.9)
\pscurve[linecolor=red](5.9,0.9)(5.9,0.6)(5.3,-0.9)

\psline[linecolor=white, linewidth=2pt](5.3,0.9)(5,-0.9)
\psline[linecolor=red](5.3,0.9)(5,-0.9)

\psdots[dotsize=2.5pt](5,0.9)(5.3,0.9)(5.6,0.9)(5.9,0.9)(6.2,0.9)(6.5,0.9)(6.8,0.9)(7.1,0.9)(7.4,0.9)(7.7,0.9)(8,0.9)(8.3,0.9)(8.6,0.9)(8.9,0.9)(9.2,0.9)(9.5,0.9)

\psdots[dotsize=2.5pt](5,-0.9)(5.3,-0.9)(5.6,-0.9)(5.9,-0.9)(6.2,-0.9)(6.5,-0.9)(6.8,-0.9)(7.1,-0.9)(7.4,-0.9)(7.7,-0.9)(8,-0.9)(8.3,-0.9)(8.6,-0.9)(8.9,-0.9)(9.2,-0.9)(9.5,-0.9)

\end{pspicture}}

\caption{The Artin braid $\beta_{x_2}$ where $x_2=((8,7,5))[6,3,1][2]$ is as in Figure~\ref{figure:ref4}. The strands corresponding to the cycle $[6,3,1]$ are drawn in blue. Note that it is a Mikado braid in $A_{B_n}$.}
\label{figure:split_braid}

\end{figure}
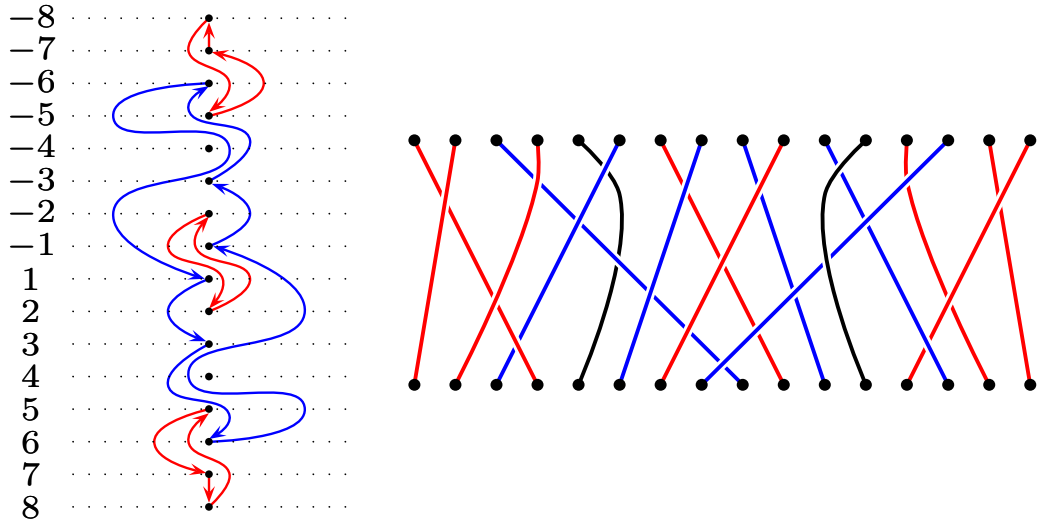

\begin{figure}[h!]

\centering
\psscalebox{1.8}{\begin{pspicture}(2,-1.92)(10,1.92)

\pscurve[linecolor=red, linewidth=0.5pt]{->}(3.5,-0.36)(3.8,-0.12)(3.4,0.12)(3.5,0.34)
\pscurve[linecolor=red, linewidth=0.5pt]{<-}(3.5,-0.34)(3.6,-0.12)(3.2,0.12)(3.5,0.36)

\pscurve[linecolor=red, linewidth=0.5pt]{->}(3.5, -0.12)(3.3, -0.36)(3.5,-0.48)(4, -0.36)(4,-0.1)(3.53,0.12)
\pscurve[linecolor=red, linewidth=0.5pt]{->}(3.5, 0.12)(3.7, 0.36)(3.5,0.48)(3, 0.36)(3,0.1)(3.47,-0.12)

\psdots[dotsize=1.6pt](3.5,1.8)(3.5,1.56)(3.5,1.32)(3.5,1.08)(3.5,0.84)(3.5,0.6)(3.5,0.36)(3.5,0.12)(3.5,-0.12)(3.5,-0.36)(3.5,-0.6)(3.5,-0.84)(3.5,-1.08)(3.5,-1.32)(3.5,-1.56)(3.5,-1.8)

\rput(2.2, 1.8){\tiny $-8$}
\rput(2.2, 1.56){\tiny $-7$}
\rput(2.2, 1.32){\tiny $-6$}
\rput(2.2, 1.08){\tiny $-5$}
\rput(2.2, 0.84){\tiny $-4$}
\rput(2.2, 0.6){\tiny $-3$}
\rput(2.2, 0.36){\tiny $-2$}
\rput(2.2, 0.12){\tiny $-1$}
\rput(2.2, -0.12){\tiny $1$}
\rput(2.2, -0.36){\tiny $2$}
\rput(2.2, -0.6){\tiny $3$}
\rput(2.2, -0.84){\tiny $4$}
\rput(2.2, -1.08){\tiny $5$}
\rput(2.2, -1.32){\tiny $6$}
\rput(2.2, -1.56){\tiny $7$}
\rput(2.2, -1.8){\tiny $8$}

\psline[linestyle=dotted, linewidth=0.4pt](2.5,1.8)(4.5,1.8)
\psline[linestyle=dotted, linewidth=0.4pt](2.5,1.56)(4.5,1.56)
\psline[linestyle=dotted, linewidth=0.4pt](2.5,1.32)(4.5,1.32)
\psline[linestyle=dotted, linewidth=0.4pt](2.5,1.08)(4.5,1.08)
\psline[linestyle=dotted, linewidth=0.4pt](2.5,0.84)(4.5,0.84)
\psline[linestyle=dotted, linewidth=0.4pt](2.5,0.6)(4.5,0.6)
\psline[linestyle=dotted, linewidth=0.4pt](2.5,0.36)(4.5,0.36)
\psline[linestyle=dotted, linewidth=0.4pt](2.5,0.12)(4.5,0.12)
\psline[linestyle=dotted, linewidth=0.4pt](2.5,-1.8)(4.5,-1.8)
\psline[linestyle=dotted, linewidth=0.4pt](2.5,-1.56)(4.5,-1.56)
\psline[linestyle=dotted, linewidth=0.4pt](2.5,-1.32)(4.5,-1.32)
\psline[linestyle=dotted, linewidth=0.4pt](2.5,-1.08)(4.5,-1.08)
\psline[linestyle=dotted, linewidth=0.4pt](2.5,-0.84)(4.5,-0.84)
\psline[linestyle=dotted, linewidth=0.4pt](2.5,-0.6)(4.5,-0.6)
\psline[linestyle=dotted, linewidth=0.4pt](2.5,-0.36)(4.5,-0.36)
\psline[linestyle=dotted, linewidth=0.4pt](2.5,-0.12)(4.5,-0.12)

\rput(2.2, 1.8){\tiny $-8$}
\rput(2.2, 1.56){\tiny $-7$}
\rput(2.2, 1.32){\tiny $-6$}
\rput(2.2, 1.08){\tiny $-5$}
\rput(2.2, 0.84){\tiny $-4$}
\rput(2.2, 0.6){\tiny $-3$}
\rput(2.2, 0.36){\tiny $-2$}
\rput(2.2, 0.12){\tiny $-1$}
\rput(2.2, -0.12){\tiny $1$}
\rput(2.2, -0.36){\tiny $2$}
\rput(2.2, -0.6){\tiny $3$}
\rput(2.2, -0.84){\tiny $4$}
\rput(2.2, -1.08){\tiny $5$}
\rput(2.2, -1.32){\tiny $6$}
\rput(2.2, -1.56){\tiny $7$}
\rput(2.2, -1.8){\tiny $8$}

\psline[linecolor=white, linewidth=2pt](6.8,0.9)(7.7,-0.9)
\psline[linecolor=red](6.8,0.9)(7.7,-0.9)

\pscurve[linecolor=white, linewidth=2pt](7.1,0.9)(7,0.6)(6.75,0)(7.4,-0.9)
\pscurve[linecolor=red](7.1,0.9)(7,0.6)(6.75,0)(7.4,-0.9)

\pscurve[linecolor=white, linewidth=2pt](7.4,0.9)(7.5,0.6)(7.75,0)(7.1,-0.9)
\pscurve[linecolor=red](7.4,0.9)(7.5,0.6)(7.75,0)(7.1,-0.9)

\psline[linecolor=white, linewidth=2pt](7.7,0.9)(6.8,-0.9)
\psline[linecolor=red](7.7,0.9)(6.8,-0.9)

\psline(5,0.9)(5,-0.9)
\psline(5.3, 0.9)(5.3,-0.9)
\psline(5.6, 0.9)(5.6,-0.9)
\psline(5.9,0.9)(5.9,-0.9)
\psline(6.2,0.9)(6.2,-0.9)
\psline(6.5, 0.9)(6.5,-0.9)

\psline(8,0.9)(8,-0.9)
\psline(8.3,0.9)(8.3,-0.9)
\psline(8.6,0.9)(8.6,-0.9)
\psline(8.9,0.9)(8.9,-0.9)
\psline(9.2,0.9)(9.2,-0.9)
\psline(9.5,0.9)(9.5,-0.9)

\psdots[dotsize=2.5pt](5,0.9)(5.3,0.9)(5.6,0.9)(5.9,0.9)(6.2,0.9)(6.5,0.9)(6.8,0.9)(7.1,0.9)(7.4,0.9)(7.7,0.9)(8,0.9)(8.3,0.9)(8.6,0.9)(8.9,0.9)(9.2,0.9)(9.5,0.9)

\psdots[dotsize=2.5pt](5,-0.9)(5.3,-0.9)(5.6,-0.9)(5.9,-0.9)(6.2,-0.9)(6.5,-0.9)(6.8,-0.9)(7.1,-0.9)(7.4,-0.9)(7.7,-0.9)(8,-0.9)(8.3,-0.9)(8.6,-0.9)(8.9,-0.9)(9.2,-0.9)(9.5,-0.9)

\end{pspicture}}

\caption{The Artin braid $\beta_t$ for $t=(1,-1)(2,-2)$ in case $i_1=2$.}
\label{figure:1_2}

\end{figure}
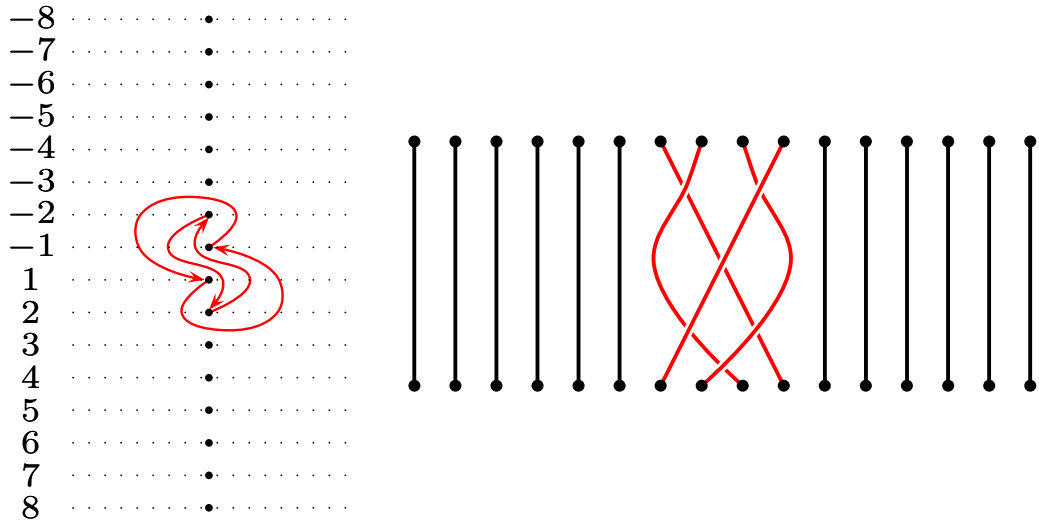

In this way, we associate to every noncrossing partition $x\in\NC(W_{D_n},c)$ an Artin braid $\beta_x\in A_{B_n}$. For some $x$ there are several possible $\beta_x\in A_{B_n}$ as illustrated in Figure~\ref{figure:ref3}, but they have the same image under $\pi_{n,1}$, hence $\pi_{n, 1}(\beta_x)$ is well-defined. We have

\begin{proposition}\label{prop:dual_diagrams}
Let $x\in W_{D_n}$, $t\in T_{D_n}$ such that $x\leq_T xt\leq_T c$. Then $$\pi_{n,1}(\beta_x \beta_{t})=\pi_{n,1}(\beta_{xt}).$$

\end{proposition}         

\begin{proof}

The situation $x\in\NC(W_{D_n},c)$, $t\in T$ and $x\leq_T xt\leq_T c$ precisely corresponds to a cover relation in the noncrossing partition lattice of type $D_n$. These covering relations were described by Athanasiadis and Reiner \cite[Section 3]{AR}: there are three families of covering relations. Setting $y=xt$, we have that $x$ is obtained from $y$ by replacing one or two balanced cycles or one paired cycle as follows: 
$$[j_1, j_2,\dots, j_k]\mapsto [j_1, \dots, j_{\ell}] ((j_{\ell+1}, \dots, j_k)), ~1\leq \ell < k \leq n-1,$$
$$((j_1, j_2,\dots, j_k))\mapsto ((j_1,\dots, j_{\ell}))((j_{\ell+1}, \dots, j_k)), ~1\leq \ell < k \leq n-1,$$
$$[j_1, \dots, j_{\ell}][j_{\ell+1}, \dots, j_k]\mapsto ((j_1, \dots, j_k)), ~1\leq \ell < k \leq n-1.$$

 Note that in the last case, we have either $\ell=1$ and $j_1=\pm i_1$ or $k=\ell+1$ and $j_k=\pm i_1$ since $x$ is a noncrossing partition. Indeed, the noncrossing partition has at most one polygon $P$ with $P=-P$, in which case the middle point lies inside $P$. 

We have to show that the braid that we obtain by the concatenation  $\beta_x\star\beta_t$ has the same image in $\widetilde{A}_{B_n}$ as $\beta_{xt}$. It is easy to deduce from the noncrossing representations $N_x$ what the result of the concatenation of two such braids is. By the process explained above, the noncrossing diagram itself can be considered as an Artin braid, viewed inside a circle or rather a cylinder. An edge of a curvilinear polygon represents a strand, and the orientation indicates the startpoint and the endpoint of that strand. 

Consider the case where the cover relation $xt\mapsto x$ is the first one above, that is, it consists of breaking a symmetric polygon into a symmetric polygon and two opposite cycles. This means that  $xt$ has the two symmetric factors $[i_1]$ and $[j_1,j_2,\dots, j_k]$ while $x$ has the same factors as $xt$ except that the two symmetric factors are replaced by $$[i_1] [j_1, \dots, j_\ell] (( j_{\ell+1}, \dots, j_k))$$

for some $\ell\in \{1, \dots, k-1\}$ and $t=((j_\ell, j_k))$. We have $k\geq 2$. All the other polygons of $x$ and $xt$ have support disjoint from $\{\pm j_1, \ldots , \pm j_k\}$, hence when concatenating $\beta_x\star \beta_t$ it is graphically clear that they will stay unchanged: indeed, these polygons are disjoint from the two curvilinear polygons associated to the reflection $t$.  Hence we can assume that $xt=[i_1][j_1,j_2,\dots, j_k]$ and $x=[i_1] [j_1, \dots, j_\ell] (( j_{\ell+1}, \dots, j_k))$. The situation is depicted in Figure~\ref{figure:cover} below. 

\begin{figure}[h!]

\begin{tabular}{cc}

\psscalebox{1.15}{\begin{pspicture}(-3,-1.92)(2.5,2.2)


\psdots(0.897,1.56)(-1.2237,1.32)(1.44,1.08)(-1.59197,0.84)(1.697,0.6)(0,0.36)(1.796,0.12)(-1.796,-0.12)(0,-0.36)(-1.697,-0.6)(1.59197,-0.84)(-1.44,-1.08)(1.2237,-1.32)(-0.897,-1.56)

\pscurve(-0.897,-1.56)(-0.5,-1.6)(-0.2,-1.6)
\pscurve{->}(0.4,-1.6)(0.7,-1.55)(1.2237,-1.32)
\psline{->}(1.2237,-1.32)(-1.44,-1.08)
\psline{->}(-1.44,-1.08)(-0.897,-1.56)

\pscurve(0.897,1.56)(0.5,1.6)(0.2,1.6)
\pscurve{->}(-0.4,1.6)(-0.7,1.55)(-1.2237,1.32)
\psline{->}(-1.2237,1.32)(1.44,1.08)
\psline{->}(1.44,1.08)(0.897,1.56)

\psline{->}(-1.697,-0.6)(1.59197,-0.84)
\psline(1.59197,-0.84)(1.72,-0.46)
\psline{->}(1.75,-0.2)(1.796,0.12)
\psline{->}(1.796,0.12)(1.697,0.6)
\psline{->}(1.697,0.6)(-1.59197,0.84)
\psline(-1.59197,0.84)(-1.72,0.46)
\psline{->}(-1.75,0.2)(-1.796,-0.12)
\psline{->}(-1.796,-0.12)(-1.697,-0.6)

\pscurve{->}(0,-0.36)(0.25,0)(0,0.36)
\pscurve{->}(0,0.36)(-0.25,0)(0,-0.36)

\pscurve[linecolor=red, linewidth=0.6pt]{->}(-1.36,-1)(0,-1.1)(1.52,-0.9)
\pscurve[linecolor=red, linewidth=0.6pt]{<-}(-1.36,-0.95)(0,-0.9)(1.32,-0.9)

\pscurve[linecolor=red, linewidth=0.6pt]{->}(1.36,1)(0,1.1)(-1.52,0.9)
\pscurve[linecolor=red, linewidth=0.6pt]{<-}(1.36,0.95)(0,0.9)(-1.32,0.9)

\rput(-0.39, 0.36){\tiny $-i_1$}
\rput(-0.3, -0.36){\tiny $i_1$}
\rput(-1.72, 0.34){\tiny ...}
\rput(1.72, -0.34){\tiny ...}
\rput(-0.1, 1.6){\tiny ...}
\rput(0.1, -1.6){\tiny ...}
\rput(1.6, 1.56){\tiny $-j_{k-1}$}
\rput(-1.9, 1.32){\tiny $-j_{\ell+1}$}
\rput(1.8, 1.08){\tiny $-j_k$}
\rput(-2.1, 0.84){\tiny $-j_\ell$}
\rput(2.2, 0.6){\tiny $j_1$}
\rput(2.2, 0.12){\tiny $j_2$}
\rput(-2.2, -0.12){\tiny $-j_2$}
\rput(-2.1, -0.6){\tiny $-j_1$}
\rput(2, -0.84){\tiny $j_\ell$}
\rput(-1.8, -1.08){\tiny $j_k$}
\rput(1.7, -1.32){\tiny $j_{\ell+1}$}
\rput(-1.5, -1.56){\tiny $j_{k-1}$}
\end{pspicture}}

&

\psscalebox{1.15}{\begin{pspicture}(-3,-1.92)(2.5,1.92)


\pscurve(-0.897,-1.56)(-0.5,-1.6)(-0.2,-1.6)
\pscurve{->}(0.4,-1.6)(0.7,-1.55)(1.2237,-1.32)
\psline{->}(-1.44,-1.08)(-0.897,-1.56)

\pscurve(0.897,1.56)(0.5,1.6)(0.2,1.6)
\pscurve{->}(-0.4,1.6)(-0.7,1.55)(-1.2237,1.32)
\psline{->}(1.44,1.08)(0.897,1.56)

\psline{->}(1.2237,-1.32)(1.59197,-0.84)
\psline{->}(-1.2237,1.32)(-1.59197,0.84)
\psline{->}(1.57,0.87)(1.44,1.08)
\psline{->}(-1.57,-0.87)(-1.44,-1.08)

\psline(1.59197,-0.84)(1.72,-0.46)
\psline{->}(1.75,-0.2)(1.796,0.12)
\psline{->}(1.796,0.12)(1.697,0.6)
\psline(-1.59197,0.84)(-1.72,0.46)
\psline{->}(-1.75,0.2)(-1.796,-0.12)
\psline{->}(-1.796,-0.12)(-1.697,-0.6)

\pscurve{->}(0,-0.36)(0.25,0)(0,0.36)
\pscurve{->}(0,0.36)(-0.25,0)(0,-0.36)

\rput(-0.39, 0.36){\tiny $-i_1$}
\rput(-0.3, -0.36){\tiny $i_1$}
\rput(-1.72, 0.34){\tiny ...}
\rput(1.72, -0.34){\tiny ...}
\rput(-0.1, 1.6){\tiny ...}
\rput(0.1, -1.6){\tiny ...}
\rput(1.61, 0.74){\tiny ...}
\rput(-1.61, -0.77){\tiny ...}

\psdots(0.897,1.56)(-1.2237,1.32)(1.44,1.08)(-1.59197,0.84)(1.697,0.6)(0,0.36)(1.796,0.12)(-1.796,-0.12)(0,-0.36)(-1.697,-0.6)(1.59197,-0.84)(-1.44,-1.08)(1.2237,-1.32)(-0.897,-1.56)

\rput(1.6, 1.56){\tiny $-j_{k-1}$}
\rput(-1.9, 1.32){\tiny $-j_{\ell+1}$}
\rput(1.8, 1.08){\tiny $-j_k$}
\rput(-2.1, 0.84){\tiny $-j_\ell$}
\rput(2.2, 0.6){\tiny $j_1$}
\rput(2.2, 0.12){\tiny $j_2$}
\rput(-2.2, -0.12){\tiny $-j_2$}
\rput(-2.1, -0.6){\tiny $-j_1$}
\rput(2, -0.84){\tiny $j_\ell$}
\rput(-1.8, -1.08){\tiny $j_k$}
\rput(1.7, -1.32){\tiny $j_{\ell+1}$}
\rput(-1.5, -1.56){\tiny $j_{k-1}$}
\end{pspicture}}

\end{tabular}

\caption{Concatenating diagrams corresponding to the cover relation $$[i_1] [j_1, \dots, j_k]\mapsto[i_1] [j_1, \dots, j_\ell] (( j_{\ell+1}, \dots, j_k).$$}
\label{figure:cover}

\end{figure}

In the concatenated diagram, the strand starting at $j_1$ first goes to $-j_\ell$ inside $\beta_x$, then the strand starting at $-j_\ell$ goes to $-j_k$ inside $\beta_t$. Hence the result is that the strand starting at $j_1$ goes to $-j_k$, and can be drawn as in the diagram on the right since there is no obstruction for such an isotopy. Similarly, the strand starting at $-j_{\ell+1}$ first goes to $-j_k$, then to $j_{\ell}$, hence is isotopic to the strand which goes directly from $-j_{\ell+1}$ to $-j_{\ell}$ as drawn in the picture on the right. The same happens on the other side, while all other strands stay unchanged. It follows that the result of the concatenation corresponds to the diagram on the right, which is precisely the diagram $N_{xt}$ associated to $xt$.

Hence we have the claim in the case where the cover relation is the one described, with $k\geq 2$. We have to show the same for the other two cover relations. We also treat the case of the last cover relation and leave the second one to the reader. Note that in the case where the cover relation is given by $$[j_1, \dots, j_{\ell}][j_{\ell+1}, \dots, j_k]\mapsto ((j_1, \dots, j_k)),$$

we have either $\ell=1$ and $j_1=\pm i_1$ or $\ell+1=k$ and $j_k=\pm i_1$. Assume that $\ell+1=k$ and $j_k=-i_1$, the case where $j_k=i_1$ as well as the cases where $\ell=1$, $j_1=\pm i_1$ are similar. We have $x=((j_1, \dots, j_\ell, -i_1))$, $t=((j_\ell, i_1))$. In this case, there are two possible diagrams $N_x$ for $x$ and the same holds for $N_t$ (see Figure~\ref{figure:ref3} for an illustration in the case where the noncrossing partition is a reflection). Since the corresponding braids $\beta_x$ obtained from the two different diagrams $N_x$ have the same image under $\pi_{n,1}$ we can choose any diagrams among the two, but the diagram $N_t$ has to be chosen to be compatible with the diagram $N_x$ if we want to do the same proof as for the first cover relation. One of the two situations is represented in Figure~\ref{figure:cover2}. Arguing as in the first case we then get the diagram on the right of the figure for the concatenation $\beta_x \star \beta_t$. This diagram is the diagram $N_{xt}$ up to the orientation of the two curves joining $i_1$ to $-i_1$: but changing their orientation corresponds to inverting a middle crossing in $\beta_{xt}$ which gives rise to a braid which has the same image in $\Ab$ thanks to the relation $\mathbf{s}_0^2=1$. This proves the claim.    

\begin{figure}[h!]

\begin{tabular}{cc}

\psscalebox{1.15}{\begin{pspicture}(-3,-1.92)(2.5,2.2)


\psline[linecolor=red, linewidth=0.6pt]{->}(-0.897,-1.5)(0, -0.41)
\pscurve[linecolor=red, linewidth=0.6pt]{<-}(-0.93,-1.5)(-0.55, -0.8)(-0.1, -0.43)

\psline[linecolor=red, linewidth=0.6pt]{->}(0.897,1.5)(0, 0.41)
\pscurve[linecolor=red, linewidth=0.6pt]{<-}(0.93,1.5)(0.55, 0.8)(0.1, 0.43)

\psdots(0.897,1.56)(-1.2237,1.32)(0.4,-1.6)(-0.4,1.6)(0,0.36)(0,-0.36)(1.2237,-1.32)(-0.897,-1.56)

\pscurve{->}(0, 0.36)(0.42,-0.3)(-0.897,-1.56)
\pscurve{->}(1.2237,-1.32)(0.65,0)(0.03,0.36)

\pscurve{->}(0, -0.36)(-0.42,0.3)(0.897,1.56)
\pscurve{->}(-1.2237,1.32)(-0.65,0)(-0.03,-0.36)

\pscurve{->}(-0.897,-1.56)(-0.5,-1.6)(-0.2,-1.6)
\pscurve{->}(0.4,-1.6)(0.7,-1.55)(1.2237,-1.32)

\pscurve{->}(0.897,1.56)(0.5,1.6)(0.2,1.6)
\pscurve{->}(-0.4,1.6)(-0.7,1.55)(-1.2237,1.32)



\rput(0.5, 0.36){\tiny $-i_1$}
\rput(-0.45, -0.36){\tiny $i_1$}
\rput(-0.1, 1.6){\tiny ...}
\rput(0.1, -1.6){\tiny ...}
\rput(1.5, 1.56){\tiny $-j_{\ell}$}
\rput(-1.8, 1.32){\tiny $-j_{1}$}
\rput(0.4, -1.88){\tiny $j_2$}
\rput(-0.4, 1.88){\tiny $-j_2$}
\rput(1.6, -1.32){\tiny $j_{1}$}
\rput(-1.4, -1.56){\tiny $j_{\ell}$}
\end{pspicture}}

&

\psscalebox{1.15}{\begin{pspicture}(-3,-1.92)(2.5,2.2)


\psdots(0.897,1.56)(-1.2237,1.32)(0,0.36)(0.4,-1.6)(-0.4,1.6)(0,-0.36)(1.2237,-1.32)(-0.897,-1.56)

\psline{->}(1.2237,-1.32)(0.897,1.56)

\psline{->}(-1.2237,1.32)(-0.897,-1.56)

\pscurve{->}(-0.897,-1.56)(-0.5,-1.6)(-0.2,-1.6)
\pscurve{->}(0.4,-1.6)(0.7,-1.55)(1.2237,-1.32)

\pscurve{->}(0.897,1.56)(0.5,1.6)(0.2,1.6)
\pscurve{->}(-0.4,1.6)(-0.7,1.55)(-1.2237,1.32)

\pscurve{<-}(0,-0.36)(0.25,0)(0,0.36)
\pscurve{<-}(0,0.36)(-0.25,0)(0,-0.36)


\rput(0.4, -1.88){\tiny $j_2$}
\rput(-0.4, 1.88){\tiny $-j_2$}
\rput(0.5, 0.36){\tiny $-i_1$}
\rput(-0.45, -0.36){\tiny $i_1$}
\rput(-0.1, 1.6){\tiny ...}
\rput(0.1, -1.6){\tiny ...}
\rput(1.5, 1.56){\tiny $-j_{\ell}$}
\rput(-1.8, 1.32){\tiny $-j_{1}$}
\rput(1.6, -1.32){\tiny $j_{1}$}
\rput(-1.4, -1.56){\tiny $j_{\ell}$}
\end{pspicture}}

\end{tabular}

\caption{Concatenating diagrams corresponding to the cover relation $$[j_1,\dots, j_{\ell}][-i_1]\mapsto ((j_1,\dots, j_\ell, -i_1)).$$}
\label{figure:cover2}

\end{figure}

\end{proof}

\begin{corollary}\label{cor:sdb}
Let $x\in\NC(W_{D_n}, c)$. Then $\pi_{n,1}(\beta_x)=x_c$. 
\end{corollary}

\begin{proof}
Recall that $S_{D_n}=\{(1,-2)(-1,2)\}\cup\{ (i, i+1)(-i,-i-1)~|~i=1,\dots, n-1\}$. By construction of the braid $\beta_t$ from the diagram $N_t$ we have that $\pi_{n,1}(\beta_s)=\mathbf{s}$ for all $s\in S_{D_n}$, and it is a general fact that $s_c=\mathbf{s}$ for every simple reflection $s$. Hence we have the claim in case $x$ is in $S_{D_n}$ and in particular $\pi_{n,1}(x)$ lies in $A_{D_n}$. Since by Proposition~\ref{prop:dual_diagrams} the elements $\pi_{n,1}(\beta_t)$ with $t\in T_{D_n}$ satisfy the dual braid relations with respect to $c$, we claim that $\pi_{n,1}(\beta_t)=t_c$ for all $t\in T_{D_n}$. Indeed, for all $t\in T_{D_n}$, we can always find $s\in S_{D_n}$ such that either $st\leq_T c$ or $ts\leq_T c$, say, $st\leq_T c$,  and $\ell_S(sts) < \ell_S(t)$ (this can be seen for instance using the noncrossing representation of $t$). It follows that we have the dual braid relation $$\pi_{n,1}(\beta_s)\pi_{n, 1}(\beta_t)=\pi_{n,1}(\beta_{sts}) \pi_{n,1}(\beta_s).$$
Arguing by induction on $\ell_S(t)$, we have that $\pi_{n,1}(\beta_q)=q_c$ for every reflection $q$ occurring in the above equality except possibly $t$. Thanks to the dual braid relation $s_c t_c= (sts)_c s_c$ we get that $\pi_{n,1}(\beta_t)=t_c$ and in particular that $\pi_{n,1}(\beta_t)\in A_{D_n}$. 

Now for $x \in NC(W_{D_n}, c)$ arbitrary we can use Proposition~\ref{prop:dual_diagrams} as well as the fact that  $x\leq_T xt \leq_T c$, $t \in T_{D_n}$, implies that $(xt)_c = x_ct_c$ (see the end of 
Subsection~\ref{Sub:DualBraid}) to get by induction on $\ell_T(x)$ that $\pi_{n,1}(\beta_x)=x_c$.
\end{proof}

\subsection{Simple dual braids are Mikado braids}\label{end}

In all the examples drawn in the figures given in the previous sections, we see that the Artin braids $\beta_x$ resulting from simple dual braids are Mikado braids: they indeed satisfy the topological condition given by the point $(2)$ of Theorem~\ref{thm:dg_b}. This is the main statement which we want to prove here. 

\begin{proposition}\label{beta_mik}
Let $x\in\NC(W_{D_n}, c)$. Then $\beta_x\in A_{B_n}$ is a Mikado braid.  
\end{proposition}

\begin{proof}

As $\beta_x \in A_{B_n}$, it suffices to verify the point $(2)$ of Theorem~\ref{thm:dg_b}. Note that except in case $x=(1,-1)(2,-2)$ and $i_1=2$ (in which case the braid $\beta_x$ which is drawn in Figure~\ref{figure:1_2} is obviously Mikado), the diagram which we obtained from $N_x$ by putting all the dots on the same vertical line (as done in Figures~\ref{figure:ref1} and \ref{figure:ref2}; we call this diagram a \textit{vertical diagram}) has the following property: each oriented curve joining two points either strictly increases or strictly decreases, and every two such distinct curves never cross. The first property follows from the fact that the diagram is obtained from $N_x$ by projecting to the right a curve which is already either strictly increasing or strictly decreasing, while the second follows from the fact that the polygons in $N_x$ do not cross. 

In such a diagram, consider a curve joining two points and going up with respect to the orientation, with no other curve lying at its right. It follows from the discussion in the paragraph above that it always exists. Every single point lying at the right of such a curve corresponds to a vertical unbraided strand in $\beta_x$ which  lies above all the other strands. Therefore, every such point can be removed in the vertical diagram, and the symmetric point lying at the left of the curve which is symmetric to the original curve can be removed simultaneously: it corresponds to removing a vertical unbraided strand lying above all the other strands in $\beta_x$, and simultaneously removing the symmetric unbraided strand lying below all the other strands, giving a new braid $\beta_x'$ lying in $A_{B_{n-1}}$ since we removed a symmetric pair of strands. After removing all such points in the vertical diagram, the original curve has nothing at its right, hence corresponds to a strand which lies above all the other strands, and we can therefore remove it, as well as its symmetric strand.  Again we obtain an element which lies in an Artin group of type $B_m$ for a smaller $m$.
Going on inductively, we can remove every strand corresponding to a curve, with a braid which stays symmetric at each step. If after removing the last curve we still have points, these correspond to vertical unbraided strands which can be removed. This concludes by Theorem~\ref{thm:dg_b}. We illustrate the above procedure in Example~\ref{ex_proof} below.

\end{proof}

Note that we could define more generally vertical diagrams (not necessarily corresponding to simple dual braids) and associate to them an Artin braid, which would therefore always be Mikado.

\begin{exple}\label{ex_proof}
We illustrate the procedure given in the proof of Proposition~\ref{beta_mik} in case $x$ is the element $x_2=((8,7,5))[6,3,1][2]$ from Figure~\ref{figure:ref4}. The vertical diagram and the braid $\beta_x$ are given in Figure~\ref{figure:split_braid}. The blue curve joining $6$ to $-1$ in the vertical diagram has no other curve lying at its right. There is only the single point $4$, which corresponds in $\beta_x$ to a strand which lies above all the others, with the symmetric strand $-4$ lying below all the others. Removing the pair of strands $4$ and $-4$, we get a symmetric braid on $14$ strands, hence in $A_{B_{7}}$. We can then remove the strand corresponding to the original curve joining $6$ to $-1$ as well as its symmetric strand, since there is no remaining strand lying above it. Going on inductively we eventually remove all pairs of strands. 
\end{exple}

As a corollary we get the main result

\begin{theorem}\label{thm:main}
Let $x\in\NC(W_{D_n},c)$. Then $x_c$ is a Mikado braid. 
\end{theorem}

\begin{proof}

By Corollary~\ref{cor:sdb} we have that $\pi_{n,1}(\beta_x)=x_c$ for every $x\in\NC(W_{D_n},c)$. But by Proposition~\ref{beta_mik}, $\beta_x$ is a Mikado braid in $A_{B_n}$. Applying Theorem~\ref{thm:mikado_bd} we get that $x_c=\pi_{n,1}(\beta_x)$ is a Mikado braid in $A_{D_n}$. 

\end{proof}

\end{document}